\documentclass[10pt, twoside]{article}
\usepackage[margin=2.54cm]{geometry}
\usepackage{authblk,amsmath,amsthm,amssymb,amsfonts}
\usepackage{graphicx,color,booktabs}
\usepackage{caption,enumerate}
\usepackage{mathtools}
\usepackage{mathrsfs}
\usepackage{float}
\usepackage{indentfirst}
\usepackage{setspace}
\usepackage{calc}
\usepackage{tikz,xcolor}

\usepackage{fourier}

\usepackage[colorlinks=true,citecolor=blue,urlcolor=blue]{hyperref}

\definecolor{gruen}{rgb}{0,0.625,0}     % semble une solution pour obtenir un vert assez sombre
\definecolor{rot}{rgb}{0.75,0,0}        % essai pour obtenir un rouge sombre
\definecolor{blau}{rgb}{0,0,0.75}       % essai pour obtenir un bleu sombre
\newcommand{\CC}{\ensuremath{\mathbb{C}}}

\newcommand{\ZZ}{\ensuremath{\mathbb{Z}}}

\newcommand{\DD}{\ensuremath{\mathbb{D}}}
\newcommand{\md}{\mathrm{d}}
\newcommand{\me}{\mathrm{e}}
\newcommand{\ve}{\varepsilon}
\newcommand{\mo}{\mathcal{O}}
\newcommand{\mi}{\mathrm{i}}
\newcommand{\sm}{\setminus}

\newtheorem{theorem}{Theorem}[section]
\newtheorem{lemma}[theorem]{Lemma}
\newtheorem{proposition}[theorem]{Proposition}
\newtheorem{corollary}[theorem]{Corollary}
\newtheorem{remark}[theorem]{Remark}
\numberwithin{equation}{section}

\usepackage{fancyhdr}
\usepackage{titlesec}
\titleformat{\section}{\bfseries}{\thesection .}{0.5em}{}
\titleformat{\subsection}{\it}{\thesubsection .}{0.5em}{}
\titleformat{\subsubsection}{\it}{\thesubsubsection .}{0.5em}{}
\titlespacing{\section}{0pt}{3ex plus 1ex minus .2ex}{3ex plus .2ex}
\titlespacing{\subsection}{0pt}{3ex plus 1ex minus .2ex}{3ex plus .2ex}
\titlespacing{\subsubsection}{0pt}{3ex plus 1ex minus .2ex}{3ex plus .2ex}

\DeclareMathOperator*{\Res}{Res}

\allowdisplaybreaks

\definecolor{lime}{HTML}{A6CE39}
\DeclareRobustCommand{\orcidicon}{%
	\begin{tikzpicture}
		\draw[lime, fill=lime] (0,0) 
		circle [radius=0.14] 
		node[white] {{\fontfamily{qag}\selectfont \tiny ID}};	\draw[white, fill=white] (-0.0625,0.095) 
		circle [radius=0.007];
	\end{tikzpicture}
	\hspace{-2mm}}
\foreach \x in {A, ..., Z}{%
	\expandafter\xdef\csname orcid\x\endcsname{\noexpand\href{https://orcid.org/\csname orcidauthor\x\endcsname}{\noexpand\orcidicon}}
}
\begin{document}
	
\title{\bf\Large Asymptotics of the Humbert functions $\Psi_1$ and $\Psi_2$}

\author[1]{Peng-Cheng Hang\orcidA{}\,}
\author[2,3]{Malte Henkel\orcidB{}\,}
\author[1]{Min-Jie Luo\orcidC{}\,\thanks{Corresponding author, Email: 
		\texttt{mathwinnie@live.com}; \texttt{mathwinnie@dhu.edu.cn}}}

\affil[1]{{\normalsize Department of Mathematics, School of Mathematics and Statistics, Donghua University,
		
		Shanghai 201620, People's Republic of China}\vspace{1mm}}
\affil[2]{{\normalsize Laboratoire de Physique et Chimie Th\'eoriques (CNRS UMR 7019),
		
		Universit\'e de Lorraine Nancy, B.P. 70239, F -- 54506 Vand\oe uvre l\`{e}s Nancy Cedex, France}\vspace{1mm}}

\affil[3]{{\normalsize Centro de F\'isica Te\'orica e Computacional, Universidade de Lisboa,
		
		Campo Grande, P--1749-016 Lisboa, Portugal}}

\date{}

\maketitle

\begin{abstract}
	
	A compilation of new results on the asymptotic behaviour of the Humbert functions $\Psi_1$ and $\Psi_2$, 
	and also on the Appell function $F_2$, is presented. 
	As a by-product, we confirm a conjectured limit which appeared recently in the study of the $1D$ Glauber-Ising model.
	We also propose two elementary asymptotic methods and confirm through 
	some illustrative examples that both methods have great potential and can be applied to a large class of problems of asymptotic analysis.
	Finally, some directions of future research are pointed out in order to suggest ideas for further study. \\
	
	\\
	
	\noindent
	{\it Keywords and phases}: asymptotic expansions, Humbert functions, Glauber-Ising model.
\end{abstract}

%%###########################################################################################################################%%
\section{Introduction}
%%###########################################################################################################################%%

When studying the limiting cases of the famous Appell hypergeometric functions $F_1$, $F_2$, $F_3$ and $F_4$, 
P. Humbert \cite{Humbert 1922} introduced seven confluent hypergeometric functions of two variables which are denoted by 
$\Phi_1$, $\Phi_2$, $\Phi_3$, $\Psi_1$, $\Psi_2$, $\Xi_1$ and $\Xi_2$. 
Humbert's pioneering contributions have inspired many important works. 
Here we want to mention Erd\'{e}lyi's work on systems of linear partial differential equations satisfied by Humbert's functions 
(see \cite{Erdelyi-1939} and \cite{Erdelyi-1940}). Actually, the study of systems satisfied by Humbert's functions remains an active area of research. 
For more recent work, the interested readers may refer to \cite{Mukai-2023} and  \cite{Shimomura-1996}.

The Humbert functions have a wide range of applications in various branches of mathematics and physics. 
Tuan \emph{et al.} \cite{Tuan-Saigo-Dinh-1996} introduced a useful integral transform with the function $\Phi_1$ 
in the kernel and showed that it is an isomorphism in the space of entire functions of exponential type. The function $\Phi_2$ 
occurs as an approximate solution to the Schr\"{o}dinger equation for the three-body Coulomb problem \cite{Gasaneo-1997}. 
The function $\Phi_3$ is usually related to the generalized Hille-Hardy formula \cite{Miyazaki-2014}. 
It also appears in Kumar's work on generalizing one of Ramanujan's transformation formula \cite{Kumar-2020}. 
The functions $\Phi_3$ and $\Psi_2$ have been found to be useful in the evaluation of the Voigt functions \cite{Srivastava-Miller-1987}. 
A multivariate generalization of $\Psi_2$ has been required in the study of non-central matrix-variate Dirichlet distributions \cite{Sanchez-Nagar-2003}. 
Belafhal and Nebdi \cite{Belafhal-Nebdi-2014} succeeded in generating a novel donut beam which they called the Humbert beam because the field distribution 
of such a beam is expressed in terms of the Humbert function $\Psi_1$. Later, Chib \emph{et al.} \cite{Chib-Khannous-Belafhal-2023} 
introduced the donut Humbert beam of type-II since its field distribution is expressed in terms of the Humbert function $\Psi_2$.

Very recently, the second author found that the Humbert function $\Psi_1$ 
appears naturally in the two-time correlator of the \textcolor{blue}{$1D$} Glauber-Ising model at temperature $T=0$ 
and as a by-product, he conjectured the following limit \cite{Henk}
\begin{equation}\label{Henkel's conjecture}
	\lim_{z\to 0^+}\me^{-\xi^2/(2z)}z^{1/2}\,\Psi_1\left[1,\frac{1}{2};\frac{3}{2},\frac{1}{2};-z,\frac{\xi^2}{2z}\right]
	=\frac{\pi}{2}\,\text{erf}\left(\frac{\xi}{\sqrt{2}}\right),
\end{equation}
where $\xi>0$ is kept fixed, $\text{erf}$ denotes the usual error function \cite{NIST} 
and the Humbert function $\Psi_1$ is defined by \cite[p. 26, Eq. (21)]{SrKa}
\begin{equation}
	\Psi_1[a,b;c,c';x,y]:=\sum_{m,n=0}^{\infty}\frac{(a)_{m+n}(b)_m}{(c)_m(c')_n}\frac{x^m}{m!}\frac{y^n}{n!} ~~~(|x|<1,~|y|<\infty)
\end{equation}
with $a,b\in\CC$ and $c,c'\notin\ZZ_{\leqslant 0}$.   
The conjectured limit \eqref{Henkel's conjecture} is of particular interest because in general, the behaviour of 
$\Psi_1\bigl[x,\frac{y}{x}\bigr]$ as $x\to 0$ 
is singular and is quite different from the usual behaviours of
\[\Psi_1[tx,y],\ \Psi_1[x,ty],\ \Psi_1[tx,ty],\quad t\to\infty,\]
which have been studied in detail in \cite{HaLu1,HaLu2,HaLu3}. 
Therefore, as a motivation for this paper, we first verify the correctness of the limit \eqref{Henkel's conjecture} 
by establishing some  asymptotic formulas for $\Psi_1[x,y]$ under the condition that
\begin{equation}\label{initial asymptotic condition}
	x\text{ or }y\text{ is small},\text{ and }0<\gamma_1\leqslant |xy|\leqslant \gamma_2<\infty.
\end{equation}

Next, we extend our study to the Humbert function $\Psi_2$, since $\Psi_2$ is a confluent form of $\Psi_1$. More precisely, we have \cite[p. 26, Eq. (26)]{SrKa}
\[
\lim_{\varepsilon\rightarrow0}\Psi_{1}[a,1/\varepsilon;c,c';\varepsilon x,y]=\Psi_2[a;c,c';x,y].
\]
The Humbert function $\Psi_2$ is defined by \cite[p. 26, Eq. (22)]{SrKa}
\begin{equation}
	\Psi_2[a;c,c';x,y]:=\sum_{m,n=0}^{\infty}\frac{\left(a\right)_{m+n}}{\left(c\right)_m\left(c'\right)_n}\frac{x^m}{m!}\frac{y^n}{n!} 
	~~~(|x|<\infty,~|y|<\infty),
\end{equation}
where $a\in\CC$ and $c,c'\notin\ZZ_{\leqslant 0}$. To the best of our knowledge, there has been no investigation on the asymptotics of $\Psi_2$. Joshi and Arya \cite[p. 499]{Josh-Arya-1982} proposed a bilateral inequality for $\Psi_2$, 
but as they themselves pointed out, the inequality is not simple and sharp. 
So we shall first study the asymptotics of $\Psi_2[x,y]$ under the condition \eqref{initial asymptotic condition} and then consider the asymptotics of $\Psi_2$ 
for one or two large arguments for completeness. 

Finally, we turn to the Appell function $F_2$, which is defined by \cite[p. 23, Eq. (3)]{SrKa}
\begin{equation}
	F_2[a,b,b';c,c';x,y]:=
	\sum_{m,n=0}^{\infty}\frac{\left(a\right)_{m+n}\left(b\right)_m(b')_n}{\left(c\right)_m\left(c'\right)_n}\frac{x^m}{m!}\frac{y^n}{n!} 
	~~~(|x|+|y|<1),
\end{equation}
where $a,b,b'\in\mathbb{C}$ and $c,c'\notin \mathbb{Z}_{\leqslant 0}$.
We note that both $\Psi_1$ and $\Psi_2$ are confluent forms of $F_2$:
\begin{align*}
	\lim_{\varepsilon\rightarrow0}
	F_2[a,b,1/\varepsilon;c,c';x,\varepsilon y]
	&=\Psi_1[a,b;c,c';x,y],\\
	\lim_{\varepsilon\rightarrow0}
	F_2[a,1/\varepsilon,1/\varepsilon;c,c';\varepsilon x,\varepsilon y]
	&=\Psi_2[a;c,c';x,y].
\end{align*}
The method that we apply to $\Psi_1$ can be used to derive the asymptotic formula of $F_2$ under the condition \eqref{initial asymptotic condition}.

Table \ref{Table 1} lists the results we prove in this paper and the corresponding methods we use.

%%++++++++++++++++++++++++++++++++++++++++++++++++++++++++++++++++++++++++++++++
\begin{table}[htbp]
	\centering
	\begin{tabular}{cccc}
		\toprule
		Function & Limiting Case & Method & Results\\
		\midrule
		$\Psi_1\bigl[x,\frac{y}{x}\bigr]$& $x\to 0$& Olver's Laplace method & 
		Theorems \ref{Thm: Psi_1[x,y/x] for small x} and \ref{Thm: Psi_1[x,y/x] for small x complete}\\[2ex]
		$\Psi_1\bigl[\frac{x}{y},y\bigr]$& $y\to 0$& Series manipulation technique & Theorem \ref{Thm: Psi_1[x/y,y] asymptotics}\\[2ex]
		$\Psi_2\bigl[x,y\bigr]$& $\Re(y)\to -\infty$& Mellin-Barnes integral & Theorem \ref{Thm: Psi_1[x,y] large y}\\[2ex]
		$\Psi_2\bigl[x,y\bigr]$ & \begin{tabular}{c}
			$\Re(y)\to +\infty$, \\
			or $x,y\to \infty$
		\end{tabular}& Olver's Laplace method & Theorem \ref{Thm: Psi2 large x and y}\\[2ex]
		$\Psi_2\bigl[x,\frac{\beta}{x}\bigr]$& $x\to \infty$& \begin{tabular}{c}
			Mellin-Barnes integral and \\
			Series manipulation technique
		\end{tabular} & Theorem \ref{Thm: Psi_2[x,beta/x] for large x} \\[2ex]
		$F_2\bigl[\frac{x}{y},y\bigr]$& $y\to 0$& Series manipulation technique & Corollary \ref{Cor: F2 small y}\\
		\bottomrule
	\end{tabular}
	\caption{Results proved and methods used.}
	\label{Table 1}
\end{table}
%%++++++++++++++++++++++++++++++++++++++++++++++++++++++++++++++++++++++++++++++

Here, we would also like to briefly mention two specific examples that also demonstrate the importance of the particular type of asymptotic behaviours 
(when one variable tends to infinity while the other remains small) studied in this paper. 
When studying approximate solution of a system of three charged particles, Gasaneo \emph{et al.} \cite{Gasaneo-2001} 
considered similar asymptotic behaviours of other hypergeometric functions. Under certain restrictions,
Carvalho e Silva and Srivastava \cite{Silva-Srivastava-2001} obtained asymptotic formulas of functions defined by general double series. 
However, the theorems given by Carvalho e Silva and Srivastava \emph{cannot} be used to derive the results obtained in our paper.

\textbf{Notation.} By $f(n,z)=\mo(a_n g(z))\,(z\in\Omega)$, or briefly $f(n,z)\ll a_n g(z)$, 
we mean that there exists a constant $K>0$ independent of the summation index $n$ and the variable $z$ such that
\[|f(n,z)|\leqslant K|a_n||g(z)|,\quad n\in\ZZ_{\geqslant 0},\,z\in\Omega.\]
In general, we use $f(z)\ll g(z)$ to mean that $f(z)=\mo(g(z))$. Moreover, the generalized hypergeometric function $_pF_q$ is defined by (see, for example, \cite{AnAR} and \cite{NIST})
\begin{equation}%\label{pFq definition}
	{}_pF_q\left[\begin{matrix}
		a_1,\cdots,a_p\\
		b_1,\cdots,b_q
	\end{matrix};z\right]
	\equiv
	{}_pF_q[
	a_1,\cdots,a_p;
	b_1,\cdots,b_q;z]
	:=\sum_{n=0}^{\infty}\frac{(a_1)_n\cdots(a_p)_n}{(b_1)_n\cdots(b_q)_n}\frac{z^n}{n!},
\end{equation}
where $a_1,\cdots,a_p\in\CC$ and $b_1,\cdots,b_q\in\CC\sm\ZZ_{\leqslant 0}$. Empty products and sums are taken as $1$ and $0$, respectively.

%%###########################################################################################################################%%
\section{Preliminaries}
%%###########################################################################################################################%%
We list some asymptotic formulas for the Kummer function ${}_1F_1$ and the modified Bessel function $I_{\nu}$, which are useful in what follows. In this section, $\delta>0$ is taken to be small.

For $z\to\infty$ in $-\pi\leqslant\arg(-z)\leqslant\pi$, we have \cite[Eq. (5.8)]{LiWo}
\begin{equation}\label{1F1 asymptotics}
	{}_1F_1\left[\begin{matrix}
		a\\
		c
	\end{matrix};z\right]\sim\frac{\Gamma(c)}{\Gamma(c-a)}
	\sum_{n=0}^{\infty}\frac{\left(a\right)_n\left(1+a-c\right)_n}{n!}\left(-z\right)^{-a-n}+\frac{\Gamma(c)}{\Gamma(a)}\,\me^z
	\sum_{n=0}^{\infty}\frac{\left(1-a\right)_n\left(c-a\right)_n}{n!}z^{a-c-n}.
\end{equation}
As $z\to\infty$ in $-\frac{\pi}{2}+\delta\leqslant \arg(z)\leqslant \frac{3\pi}{2}-\delta$, we have \cite[Eq. (10.40.5)]{NIST}
\begin{equation}\label{I_nu for large z}
	I_{\nu}(z)=\frac{\me^z}{\sqrt{2\pi z}}\left(1+\mo\left(z^{-1}\right)\right)+\mi\,
	\me^{\pi\mi\nu}\frac{\me^{-z}}{\sqrt{2\pi z}}\left(1+\mo\left(z^{-1}\right)\right),
\end{equation}
where
\begin{equation}\label{I_nu - Def}
	I_\nu(z)=\left(\frac{1}{2}z\right)^{\nu}\sum_{k=0}^{\infty}\frac{(\frac{1}{4}z^2)^k}{k!\,\Gamma(\nu+k+1)}.
\end{equation}
As $a\to\infty$ in $\left|\arg(a)\right|\leqslant \pi-\delta$, we have \cite[Eq. (13.8.12)]{NIST}
\begin{equation}\label{1F1 for large a--1}
	{}_1F_1\left[\begin{matrix}
		a\\
		b
	\end{matrix};z\right]=\Gamma\left(b\right){\me^{\frac{1}{2}z}}\left(\frac{z}{a}\right)^{\frac{1}{2}(1-b)}\frac{\Gamma(a+1-b)}{\Gamma(a)}
	\left\{I_{b-1}\left(2\sqrt{az}\right){\left(1+\mo\left(a^{-1}\right)\right)}-\sqrt{\frac{z}{a}}\,I_b\left(2\sqrt{az}\right){\left(1+\mo\left(a^{-1}\right)\right)}\right\}.
\end{equation}

Choose $a$ and $z$ such that $\left|\arg(a)\right|\leqslant \pi-\delta$ and $0\leqslant\arg(z)\leqslant 2\pi$. Then
\[-\frac{\pi}{2}+\frac{\delta}{2}\leqslant \arg(\sqrt{az})\leqslant \frac{3\pi}{2}-\frac{\delta}{2}.\]
It follows from \eqref{I_nu for large z} that as $a\to\infty$ with $z\ne 0$ fixed,
\[I_{\nu}(2\sqrt{az})=\frac{1}{2\sqrt{\pi}}\left(az\right)^{-\frac{1}{4}}\left(\me^{2\sqrt{az}}\bigl(1+o(1)\bigr)
+\mi\,\me^{\pi\mi\nu}\me^{-2\sqrt{az}}\bigl(1+o(1)\bigr)\right).\]
Using \eqref{1F1 for large a--1}, we obtain that for large $a$ in $\left|\arg(a)\right|\leqslant \pi-\delta$ 
and fixed $z\ne 0$ in $0\leqslant\arg(z)\leqslant 2\pi$,
\begin{equation}\label{1F1 for large a--2}
	{}_1F_1\left[\begin{matrix}
		a\\
		b
	\end{matrix};z\right]=\Gamma\left(b\right)\frac{\Gamma(a+1-b)}{2\sqrt{\pi}\,\Gamma(a)}z^{\frac{1}{4}-\frac{1}{2}b}{\me^{\frac{1}{2}z}}a^{\frac{1}{2}b-\frac{3}{4}}
	\left(\me^{2\sqrt{az}}\bigl(1+o(1)\bigr)-\mi\,\me^{\pi\mi b}\me^{-2\sqrt{az}}\bigl(1+o(1)\bigr)\right).
\end{equation}

Finally, we shall also need the Euler-type integral \cite[Eq. (13.4.1)]{NIST}
\begin{equation}\label{1F1 Euler-type integral}
	{}_1F_1\left[\begin{matrix}
		a\\
		b
	\end{matrix};z\right]=\frac{\Gamma(b)}{\Gamma(a)\Gamma(b-a)}\int_0^1 t^{a-1}\left(1-t\right)^{b-a-1}\me^{zt}\md t,\quad \Re(b)>\Re(a)>0.
\end{equation}

%%###########################################################################################################################%%
\section{Humbert function $\Psi_1$}
%%###########################################################################################################################%%

First of all, we list a couple of identities which are useful in what follows:
\begin{subequations}
	\begin{align}
		\Psi_1[a,b;c,c';x,y]& =\sum_{n=0}^{\infty}\frac{\left(a\right)_k\left(b\right)_k}{\left(c\right)_k}\,_1F_1\left[\begin{matrix}
			a+k\\
			c'
		\end{matrix};y\right]\frac{x^k}{k!} \label{Psi_1 expansion in 1F1}\\
		& =\frac{\Gamma(c)}{\Gamma(b)\Gamma(c-b)}\int_0^1 t^{b-1}\left(1-t\right)^{c-b-1}\left(1-xt\right)^{-a}{}_1F_1\left[\begin{matrix}
			a\\
			c'
		\end{matrix};\frac{y}{1-xt}\right]\md t \label{Psi_1 Euler integral}\\
		& =\frac{1}{\Gamma(a)}\int_0^{\infty}u^{a-1}\me^{-u}\,_1F_1\left[\begin{matrix}
			b\\
			c
		\end{matrix};xu\right]{}_0F_1\left[\begin{matrix}
			-\\
			c'
		\end{matrix};yu\right]\md u. \label{Psi_1 Laplace integral}
	\end{align}
\end{subequations}
With the exception of \eqref{Psi_1 expansion in 1F1}, they arise in several places in the literature (see \cite{BrSa}). 
For instance, \eqref{Psi_1 Euler integral} is useful for numerical computations (see \cite[Appendix A]{HaLu1}), 
and \eqref{Psi_1 Laplace integral} is from \cite{WaHe}. 
The convergence conditions for these identities are given below, respectively,
\begin{align*}
	\eqref{Psi_1 expansion in 1F1}:& \quad c,c'\notin\ZZ_{\leqslant 0},\ |x|<1,|y|<\infty;\\
	\eqref{Psi_1 Euler integral}:& \quad \Re(c)>\Re(b)>0,\ c'\notin\ZZ_{\leqslant 0},\ x\notin[1,+\infty),\,y\in\CC;\\
	\eqref{Psi_1 Laplace integral}: & \quad \Re(a)>0,\ c,c'\notin\ZZ_{\leqslant 0},\ \Re(x)<1,\,y\in\CC.
\end{align*}
Moreover, $\Psi_1$ has an extension to the region (see \cite[Section 2.2]{HaLu3})
\[\DD_{\Psi_1}:=\bigl\{(x,y)\in\CC^2\colon x\ne 1,\ |{\arg}(1-x)|<\pi,\ |y|<\infty\bigr\}.\]

Before giving the first asymptotic result of $\Psi_1$ under the condition \eqref{initial asymptotic condition} 
(see Theorem \ref{Thm: Psi_1[x,y/x] for small x}), we derive a useful integral representation.

\begin{proposition}
	If $\Re(a)>0,\,b\in\CC,\,c,c'\in\CC\sm\ZZ_{\leqslant 0},\,\Re(x)>0$ and $\Re\big(\frac{y}{x}\big)>0$, then
	\begin{equation}\label{Psi_1[-x,y^2/x] integral}
		\Psi_1\left[a,b;c,c';-x,\frac{y^2}{x}\right]=
		\frac{2\,\Gamma(c')}{\Gamma(a)}x^{-a}\left(\frac{y}{x}\right)^{1-c'}\int_0^{\infty}w^{2a-c'}\me^{-\frac{w^2}{x}}\,_1F_1\left[\begin{matrix}
			b\\
			c
		\end{matrix};-w^2\right]I_{c'-1}\left(2\frac{y}{x}w\right)\md w.
	\end{equation}
\end{proposition}

\begin{proof}
	Assume for the moment that $a>0,\,x>0$ and $y>0$. In view of the relation
	\cite[Eq. (7.13.1.1)]{Integrals & Series v3}
	\begin{equation}\label{0F1 identity}
		{}_0F_1\left[\begin{matrix}
			-\\
			b
		\end{matrix};z\right]=\Gamma\left(b\right)z^{\frac{1-b}{2}}I_{b-1}\left(2\sqrt{z}\right)
	\end{equation}
	and the integral representation \eqref{Psi_1 Laplace integral}, we obtain
	\begin{align*}
		\Psi_1\left[a,b;c,c';-x,\frac{y^2}{x}\right]& =\frac{1}{\Gamma(a)}\int_0^{\infty}u^{a-1}\me^{-u}\,_1F_1\left[\begin{matrix}
			b\\
			c
		\end{matrix};-xu\right]{}_0F_1\left[\begin{matrix}
			-\\
			c'
		\end{matrix};\frac{y^2}{x}u\right]\md u\\
		& =\frac{\Gamma(c')}{\Gamma(a)}\int_0^{\infty}u^{a-1}\me^{-u}\,_1F_1\left[\begin{matrix}
			b\\
			c
		\end{matrix};-xu\right]\left(\frac{y^2}{x}u\right)^{\frac{1-c'}{2}}I_{c'-1}\left(2\sqrt{\frac{y^2}{x}u}\right)\md u\\
		& =\frac{\Gamma(c')}{x\,\Gamma(a)}\int_0^{\infty}\left(\frac{v}{x}\right)^{a-1}\me^{-\frac{v}{x}}\,_1F_1\left[\begin{matrix}
			b\\
			c
		\end{matrix};-v\right]\left(\frac{y^2}{x^2}v\right)^{\frac{1-c'}{2}}I_{c'-1}\left(2\frac{y}{x}\sqrt{v}\right)\md v\\
		& =\frac{2\,\Gamma(c')}{\Gamma(a)}x^{-a}\left(\frac{y}{x}\right)^{1-c'}\int_0^{\infty}w^{2a-c'}\me^{-\frac{w^2}{x}}\,_1F_1\left[\begin{matrix}
			b\\
			c
		\end{matrix};-w^2\right]I_{c'-1}\left(2\frac{y}{x}w\right)\md w.
	\end{align*}
	The proof is then completed by using \eqref{I_nu for large z} and analytic continuation.
\end{proof}

\begin{theorem}\label{Thm: Psi_1[x,y/x] for small x}
	Let $\Re(a)>0,\,b\in\CC$ and $c,c'\in\CC\sm\ZZ_{\leqslant 0}$. Then for fixed $y>0$,
	\begin{equation}\label{Psi_1[x,y/x] asymptotics--1}
		\Psi_1\left[a,b;c,c';\pm x,\frac{y}{x}\right]\sim \frac{\Gamma(c')}{\Gamma(a)}\,_1F_1\left[\begin{matrix}
			b\\
			c
		\end{matrix};\pm y\right]\left(\frac{y}{x}\right)^{a-c'}\me^{\frac{y}{x}}
	\end{equation}
	as $x\to 0$ in $\left|\arg(x)\right|\leqslant\frac{\pi}{2}-\delta$, where $\delta\in\bigl(0,\frac{\pi}{2}\bigr]$ is fixed.
\end{theorem}

\begin{proof}
	We only give the proof of the result for  $\Psi_1\bigl[-x,\frac{y^2}{x}\bigr]$. 
	To establish the result for $\Psi_1\bigl[+x,\frac{y^2}{x}\bigr]$, one may repeat the proof here by starting with
	\[\Psi_1\left[a,b;c,c';+x,\frac{y^2}{x}\right]=\frac{2\,\Gamma(c')}{\Gamma(a)}x^{-a}\left(\frac{y}{x}\right)^{1-c'}
	\int_0^{\infty}w^{2a-c'}\me^{-\frac{w^2}{x}}\,_1F_1\left[
	\begin{matrix}
		b\\
		c
	\end{matrix}
	;+w^2\right]I_{c'-1}\left(2\frac{y}{x}w\right)\md w.\]
	
	Assume that $\Re(a)>0,\,\Re(x)>0$ and $y>0$. Then separate the integral in \eqref{Psi_1[-x,y^2/x] integral} into two parts, namely,
	\[\Psi_1:=\Psi_1\left[a,b;c,c';-x,\frac{y^2}{x}\right]=\frac{2\,\Gamma(c')}{\Gamma(a)}x^{-a}\left(\frac{y}{x}\right)^{1-c'}
	\left(\mathfrak{I}_1+\mathfrak{I}_2\right),\]
	where $\eta=\bigl|\frac{x}{y}\bigr|$ and
	\begin{align*}
		\mathfrak{I}_1& =\int_0^{\eta}w^{2a-c'}\me^{-\frac{w^2}{x}}\,_1F_1\left[
		\begin{matrix}
			b\\
			c
		\end{matrix}
		;-w^2\right]I_{c'-1}\left(2\frac{y}{x}w\right)\md w,\\
		\mathfrak{I}_2& =\int_{\eta}^{\infty}w^{2a-c'}\me^{-\frac{w^2}{x}}\,_1F_1\left[
		\begin{matrix}
			b\\
			c
		\end{matrix}
		;-w^2\right]
		I_{c'-1}\left(2\frac{y}{x}w\right)\md w.
	\end{align*}
	Note that $\eta=\mo(x)$ as $x\to 0$. Thus, by \eqref{I_nu - Def}, we have
	\begin{align*}
		\left|\mathfrak{I}_1\right|& \leqslant\int_0^{\eta}w^{\Re(2a-c')}\me^{-\frac{w^2}{\Re(x)}}\left|{}_1F_1\left[
		\begin{matrix}
			b\\
			c
		\end{matrix}
		;-w^2\right]\right|\cdot\left|I_{c'-1}\left(2\frac{y}{x}w\right)\right|\md w\\
		& \ll\int_0^{\eta}w^{\Re(2a-c')}\left|\frac{y}{x}w\right|^{\Re(c')-1}\cdot\left|\sum_{k=0}^{\infty}\frac{(wy/x)^{2k}}{k!\,\Gamma(c'+k)}\right|\md w\\
		& \ll\eta^{1-\Re(c')}\int_0^{\eta}w^{\Re(2a)-1}\md w=\frac{1}{2\Re(a)}\eta^{\Re(2a-c')+1}\ll \left|x\right|^{\Re(2a-c')+1}.
	\end{align*}
	To estimate $\mathfrak{I}_2$, recall $\eta=\bigl|\frac{x}{y}\bigr|$ and use \eqref{I_nu for large z} to yield
	\[I_{c'-1}\left(2\frac{y}{x}w\right)=\frac{\me^{2yw/x}}{\sqrt{2\pi\cdot 2yw/x}}\left(1+\mo(\eta w^{-1})\right),\quad w\geqslant \eta.\]
	Therefore,
	\[
	\mathfrak{I}_2=\frac{1}{2\sqrt{\pi}}\left(\frac{x}{y}\right)^{\frac{1}{2}}\me^{\frac{y^2}{x}}\int_{\eta}^{\infty}w^{2a-c'-\frac{1}{2}}\,
	\me^{-\frac{1}{x}(w-y)^2}
	\,_1F_1\left[
	\begin{matrix}
		b\\
		c
	\end{matrix}
	;-w^2
	\right]\left(1+\mo(\eta w^{-1})\right)\md w.
	\]
	Using Olver's Laplace method \cite[Theorem I]{Olve}, we obtain
	\[
	\mathfrak{I}_2\sim \frac{1}{2}xy^{2a-c'-1}\me^{\frac{y^2}{x}}
	\,_1F_1\left[
	\begin{matrix}
		b\\
		c
	\end{matrix}
	;-y^2\right].
	\]
	Hence \eqref{Psi_1[x,y/x] asymptotics--1} follows from the estimates on $\mathfrak{I}_1$ and $\mathfrak{I}_2$.
\end{proof}

\begin{remark}
	The limit \eqref{Henkel's conjecture} is a direct result of \eqref{Psi_1[x,y/x] asymptotics--1}. Indeed, \eqref{Psi_1[x,y/x] asymptotics--1} implies that
	\[\lim_{z\to 0^+}\me^{-\xi^2/(2z)}z^{1/2}\,\Psi_1\left[1,\frac{1}{2};\frac{3}{2},\frac{1}{2};-z,\frac{\xi^2}{2z}\right]
	=\sqrt{\frac{\pi}{2}}\,\xi\cdot{}_1F_1\left[
	\begin{matrix}
		1/2
		\\
		3/2
	\end{matrix}~;-\frac{\xi^2}{2}\right],
	\]
	which, in view of {\rm\cite[Eq. (7.11.1) and (13.6.5)]{NIST}}, is equivalent to \eqref{Henkel's conjecture}.
\end{remark}

We can further establish the complete leading-term behaviour of $\Psi_1\bigl[x,\frac{y}{x}\bigr]$ for small $x$, 
but we have to impose different conditions on the parameters.

\begin{theorem}\label{Thm: Psi_1[x,y/x] for small x complete}
	Let $a\in\CC,\,c,c'\in\CC\sm\ZZ_{\leqslant 0}$ and $\Re(c)>\Re(b)>0$. For any fixed $y\in\CC\sm\{0\}$,
	\begin{equation}
		\Psi_1\left[a,b;c,c';x,\frac{y}{x}\right]\sim \frac{\Gamma(c')}{\Gamma(a)}\,_1F_1\left[\begin{matrix}
			b\\
			c
		\end{matrix};y\right]\left(\frac{y}{x}\right)^{a-c'}\me^{\frac{y}{x}}+\frac{\Gamma(c')}{\Gamma(c'-a)}\left(-\frac{y}{x}\right)^{-a}
	\end{equation}
	as $x\to 0$ such that $-\pi\leqslant \arg\bigl(-\frac{y}{x}\bigr)\leqslant \pi$.
\end{theorem}

\begin{proof}
	Fix $y\in\CC\sm\{0\}$ and write
	\[\lambda=\lambda(u)\equiv\frac{y/x}{1+ux}.\]
	Then as $x\to 0$, $\lambda=\frac{y}{x}-yu+\mo(x)$ and $\lambda^{-1}=\mo(x)$. Use of \eqref{Psi_1 Euler integral} gives
	\begin{equation}\label{integral for Psi}
		\Psi_1\left[a,b;c,c';-x,\frac{y}{x}\right]=\frac{\Gamma(c)}{\Gamma(b)\Gamma(c-b)}
		\int_0^1 u^{b-1}\left(1-u\right)^{c-b-1}\left(1+ux\right)^{-a}g(\lambda)\md u,
	\end{equation}
	where
	\[g(\lambda):={}_1F_1\left[\begin{matrix}
		a\\
		c'
	\end{matrix};\lambda\right]={}_1F_1\left[\begin{matrix}
		a\\
		c'
	\end{matrix};\frac{y}{x}-yu+\mo(x)\right].\]
	Applying the expansion \eqref{1F1 asymptotics}, we deduce that as $x\to 0$ in $-\pi\leqslant \arg\left(-\frac{y}{x}\right)\leqslant \pi$,
	\[
	g(\lambda)=\frac{\Gamma(c')}{\Gamma(a)}\left(\frac{y}{x}\right)^{a-c'}\me^{\frac{y}{x}}\left(\mathrm{e}^{-yu}+
	\mo(x)\right)+\frac{\Gamma(c')}{\Gamma(c'-a)}\left(-\frac{y}{x}\right)^{-a}\left(1+\mo(x)\right).
	\]
	Combing this with \eqref{integral for Psi} and \eqref{1F1 Euler-type integral} completes the proof.
\end{proof}

\begin{remark}\label{Remark: asymptotics of F_3 and Humbert functions}
	With the help of the integrals given by Brychkov and Saad {\rm\cite[Eq. (3.2) and (3.15)-(3.18)]{BrSa1}},  
	one could try to find asymptotic formulas for the Appell function $F_3$ and the Humbert functions $\Phi_2,\Phi_3,\Xi_1,\Xi_2$ 
	under the condition \eqref{initial asymptotic condition}. We plan to return to this elsewhere. 
\end{remark}

Since $\Psi_1[x,y]$ is not symmetric in $x$ and $y$, we are led to study the behaviour of $\Psi_1\bigl[\frac{x}{y},y\bigr]$ for small $y$.

\begin{theorem}\label{Thm: Psi_1[x/y,y] asymptotics}
	Let $a,b\in\CC,\,c,c'\in\CC\sm\ZZ_{\leqslant 0}$ and $a-b\notin\ZZ$. For fixed $x\in\CC\sm\{0\}$,
	\begin{align}
		\Psi_1\left[a,b;c,c';\frac{x}{y},y\right]={}& \mathfrak{f}_c(b,a)\left(-\frac{x}{y}\right)^{-a}\,
		\sum_{m=0}^{N-1}a_1(m)y^m+\mathfrak{f}_c(a,b)\left(-\frac{x}{y}\right)^{-b}\,\sum_{m=0}^{N-1}a_2(m)y^m  \\
		& +\mo\left(\left|y\right|^{\Re(a)+N}+\left|y\right|^{\Re(b)+N}\right) \nonumber 
	\end{align}
	as $y\to 0$ in $\bigl|\arg\bigl(-\frac{x}{y}\bigr)\bigr|<\pi$, where $N$ is any positive integer,
	\begin{subequations}
		\begin{equation}\label{f_gamma(a,b) definition}
			\mathfrak{f}_{\gamma}(a,b):=\frac{\Gamma(\gamma)\Gamma(a-b)}{\Gamma(a)\Gamma(\gamma-b)}
		\end{equation}
		and
		\begin{align}
			a_1(m)& =\sum_{\substack{k,\ell\geqslant 0\\k+\ell=m}}
			\frac{\left(-k\right)_{\ell}\left(a\right)_k\left(1-c+a\right)_k}{\left(c'\right)_{\ell}\left(1-b+a\right)_k \ell!\,k!}x^{-k},\\
			a_2(m)& =\sum_{\substack{k,\ell\geqslant 0\\k+\ell=m}}
			\frac{\left(a-b-k\right)_{\ell}\left(b\right)_k\left(1-c+b\right)_k}{\left(c'\right)_{\ell}\left(1-a+b\right)_k \ell!\,k!}x^{-k}.
		\end{align}
	\end{subequations}
\end{theorem}

\begin{proof}
	Recall the series expansion \cite[Theorem 3.8]{HaLu3}: for $\left|\arg(-x)\right|<\pi,\,|x|>1$ and $|y|<\infty$,
	\begin{equation}\label{Psi_1 series for |x|>1}
		\begin{split}
			\Psi_1[a,b;c,c';x,y]={}& \mathfrak{f}_c(b,a)\left(-x\right)^{-a}\sum_{k=0}^{\infty}{}_1F_1\left[\begin{matrix}
				-k\\
				c'
			\end{matrix};y\right]\frac{\left(a\right)_k\left(1-c+a\right)_k}{\left(1-b+a\right)_k k!}\frac{1}{x^k}\\
			&+\mathfrak{f}_c(a,b)\left(-x\right)^{-b}\sum_{k=0}^{\infty}{}_1F_1\left[\begin{matrix}
				a-b-k\\
				c'
			\end{matrix};y\right]\frac{\left(b\right)_k\left(1-c+b\right)_k}{\left(1-a+b\right)_k k!}\frac{1}{x^k}.
		\end{split}
	\end{equation}
	Denote the series on the right by $U_1(x,y)$ and $U_2(x,y)$, respectively. Then we have
	\begin{align}
		U_1\left(\frac{x}{y},y\right)& =\sum_{k=0}^{N-1}\sum_{\ell=0}^{\infty}\frac{\left(-k\right)_{\ell}}{\left(c'\right)_{\ell}}\frac{y^{\ell}}{\ell!}\cdot 
		\frac{\left(a\right)_k\left(1-c+a\right)_k}{\left(1-b+a\right)_k k!}\frac{y^k}{x^k}+\mo\left(\left|y\right|^N\right) \nonumber\\
		& =\sum_{m=0}^{N-1}a_1(m)y^m+\mo\left(\left|y\right|^N\right), \label{U_1 expansion}
	\end{align}
	and similarly
	\begin{equation}\label{U_2 expansion}
		U_2\left(\frac{x}{y},y\right)=\sum_{m=0}^{N-1}a_2(m)y^m+\mo\left(\left|y\right|^N\right).
	\end{equation}
	The expansion for $\Psi_1\big[\frac{x}{y},y\big]$ therefore follows from \eqref{Psi_1 series for |x|>1}-\eqref{U_2 expansion}.
\end{proof}

\begin{remark} \label{rem3.7} 
	As an application, notice that for $x\to\infty$ and $y>0$ fixed, one has
	\[\Psi_1\left[1,\frac{1}{2};\frac{3}{2},\frac{1}{2};-x,\frac{y}{x}\right]\sim \frac{\pi}{2}x^{-\frac{1}{2}}-x^{-1}+\mo\left(x^{-\frac{3}{2}}\right).\]
	This new asymptotic identity, especially the $y$-independence of the leading two terms, proves 
	the large-time universality in the time-space correlator in the $1D$ Glauber-Ising model at temperature $T=0$ {\rm\cite{Henk}}.
\end{remark}
%%###########################################################################################################################%%
\section{Humbert function $\Psi_2$}
%%###########################################################################################################################%%

Clearly, $\Psi_2$ satisfies the symmetry relation $\Psi_2[a;c,c';x,y]=\Psi_2[a;c',c;y,x]$ and admits the series expansion
\begin{equation}\label{Psi_2 series expansion--1}
	\Psi_2[a;c,c';x,y]=\sum_{n=0}^{\infty}\frac{\left(a\right)_n}{\left(c'\right)_n}\,_1F_1\left[\begin{matrix}
		a+n\\
		c
	\end{matrix};x\right]\frac{y^n}{n!},\quad |x|<\infty,|y|<\infty.
\end{equation}
Using the Kummer transformation \cite[Eq. (13.2.39)]{NIST}, we get
\begin{equation}\label{Psi_2 series expansion--2}
	\Psi_2[a;c,c';x,y]=\me^x\sum_{n=0}^{\infty}\frac{\left(a\right)_n}{\left(c'\right)_n}\,_1F_1\left[\begin{matrix}
		c-a-n\\
		c
	\end{matrix};-x\right]\frac{y^n}{n!},\quad |x|<\infty,|y|<\infty.
\end{equation}

Let us first study the asymptotic behaviour of $\Psi_2[x,y]$ as $y\rightarrow\infty$ in the left half-plane using the Mellin-Barnes integral representation of $\Psi_2$.  

\begin{theorem}\label{Thm: Psi_1[x,y] large y}
	{\rm(i)} Let $a\in\CC,\,c,c'\in\CC\sm\ZZ_{\leqslant 0},\,\delta\in\bigl(0,\frac{\pi}{2}\bigr]$ and
	\[
	\mathbb{V}_{\Psi_2}=\left\{(x,y)\in\CC^2:x\ne 0,\ y\ne 0,\ 0\leqslant \arg(x)\leqslant 2\pi,\ \left|\arg(-y)\right| \leqslant \frac{\pi}{2}-\delta \right\}.
	\]
	Then
	\begin{subequations}
		\begin{align}\label{Psi_2 Mellin-Barnes integral}
			\Psi_2[a;c,c';x,y] &=\frac{1}{2\pi\mi}\frac{\Gamma(c')}{\Gamma(a)}\int_{\mathfrak{L}_{\mi\sigma\infty}}{}_1F_1\left[\begin{matrix}
				a+s\\
				c
			\end{matrix};x\right]\frac{\Gamma(a+s)}{\Gamma(c'+s)}\Gamma(-s)\left(-y\right)^s\md s,
		\end{align}
		where the path $\mathfrak{L}_{\mi\sigma\infty}$, starting at $\sigma-\mi\infty$ and ending at $\sigma+\mi\infty$, 
		is a vertical line intended if necessary to separate the poles of $\Gamma(a+s)$ from the poles of $\Gamma(-s)$.
		
		{\rm(ii)} Let $a\in\CC$ and $c,c',c'-a\in\CC\sm\ZZ_{\leqslant 0}$. When $(x,y)\in\mathbb{V}_{\Psi_2}$ and $|y|\to\infty$, we have
		\begin{align}
			\Psi_2[a;c,c';x,y] &\sim\frac{\Gamma(c')}{\Gamma(c'-a)}\sum_{n=0}^{\infty}{}_1F_1\left[\begin{matrix}
				-n\\
				c
			\end{matrix};x\right]\frac{\left(a\right)_n(1+a-c')_n}{n!}(-y)^{-a-n}.
		\end{align}
	\end{subequations}
\end{theorem}
\begin{proof}
	(i) The proof will be divided into two steps. Firstly, we prove the integral (\ref{Psi_2 Mellin-Barnes integral}) exists for $(x,y)\in\mathbb{V}_{\Psi_2}$. 
	Secondly, we show by using Cauchy's residue theorem that the integral is equal to $\Psi_2$ when $x$ and $y$ are appropriately restricted. 
	Then by analytic continuation, \eqref{Psi_2 Mellin-Barnes integral} is valid throughout $\mathbb{V}_{\Psi_2}$.
	
	\textbf{Step 1.}	We can formally obtain \eqref{Psi_2 Mellin-Barnes integral} from \eqref{Psi_2 series expansion--1} 
	with the help of Ramanujan's master theorem \cite[p. 107]{AEGHH}. To check the validity, we first denote the integrand by
	\begin{equation}\label{Psi(s) definition}
		\Psi(s):={}_1F_1\left[\begin{matrix}
			a+s\\
			c
		\end{matrix};x\right]\frac{\Gamma(a+s)}{\Gamma(c'+s)}\Gamma(-s)(-y)^s.
	\end{equation}
	Recall the estimate \cite[Eq. (3.4)]{HaLu3}: as $t\to\pm \infty$,
	\begin{equation}\label{estimate for Psi(s)--1}
		\left|\frac{\Gamma(a+\sigma+\mi t)}{\Gamma(c'+\sigma+\mi t)}\Gamma(-\sigma-\mi t)(-y)^{\sigma+\mi t}\right|
		=\mo\left(\left|t\right|^{\Re(a-c')-\sigma-\frac{1}{2}}\me^{-\left(\frac{\pi}{2}\pm\arg(-y)\right)|t|}\right).
	\end{equation}
	Using \eqref{1F1 for large a--2}, we obtain that as $t\to\pm \infty$,
	\begin{equation}\label{estimate for Psi(s)--2}
		\left|{}_1F_1\left[\begin{matrix}
			a+\sigma+\mi t\\
			c
		\end{matrix};x\right]\right|=\mo\left(\left|t\right|^{\frac{1}{4}-\frac{1}{2}\Re(c)}\me^{3\sqrt{|xt|}}\right).
	\end{equation}
	Hence when $\left|\arg(-y)\right|\leqslant \frac{\pi}{2}-\delta$, the integrand $\Psi(s)$ decays exponentially as 
	$|s|\to\infty$ along the path $\mathfrak{L}_{\mi\sigma\infty}$. Therefore, the integral in \eqref{Psi_2 Mellin-Barnes integral} is certainly convergent in $\mathbb{V}_{\Psi_2}$.
	
	\textbf{Step 2.} Let us show that in a region $\widetilde{\mathbb{V}}_{\Psi_2}\subset \mathbb{V}_{\Psi_2}$, 
	the integral in \eqref{Psi_2 Mellin-Barnes integral} yields the series representation \eqref{Psi_2 series expansion--1} of $\Psi_2$.  Define
	\[\widetilde{\mathbb{V}}_{\Psi_2}
	=\left\{(x,y)\in\CC^2:x\ne 0, \ 0<|y|<\frac{1}{\mathrm{e}},\ 0\leqslant \arg(x)\leqslant 2\pi,\ \left|\arg(-y)\right| 
	\leqslant \frac{\pi}{2}-\delta \right\}.\]
	We close the path of integration by the semi-circle $\mathfrak{C}$, that is parameterized by 
	$s=(N+1/2)\me^{\mi\theta}$ $(|\theta|\leqslant \pi/2)$, and then we let $N\to\infty$ through integral values. Recall the estimate \cite[p. 7]{HaLu3}
	\[\left|\frac{\Gamma(a+s)}{\Gamma(c'+s)}\Gamma(-s)(-y)^s\right|
	=\left(N^{\Re(a-c')-\frac{1}{2}-\left(N+\frac{1}{2}\right)\cos\theta}\me^{\left(N+\frac{1}{2}\right)(\delta_1\cos\theta+\delta_2|\sin\theta|)}\right),
	\quad s\to\infty,\,s\in\mathfrak{C},\]
	where $\delta_1:=1+\log|y|<0$ and $\delta_2:=|\theta|+\left|\arg(-y)\right|-\pi \leqslant-\delta<0$. 
	It follows from \eqref{1F1 for large a--2} that
	\[\left|{}_1F_1\left[\begin{matrix}
		a+s\\
		c
	\end{matrix};x\right]\right|=\mo\left(N^{\frac{1}{4}-\frac{1}{2}\Re(c)}\me^{3\sqrt{N|x|}}\right),\quad s\to\infty,\,s\in\mathfrak{C}.\]
	So the integral on the semi-circle $\mathfrak{C}$ tends to zero as $N\to\infty$. 
	Finally, we get the representation \eqref{Psi_2 series expansion--1} from Cauchy's residue theorem and the assertion that (\cite[p. 7]{AnAR})
	\[\underset{s=n}{\Res}\ \Gamma(-s)=\frac{\left(-1\right)^{n-1}}{n!}.\]
	
	(ii) Let $I$ denote the right-hand side of \eqref{Psi_2 Mellin-Barnes integral}, and let $\Psi(s)$ denote the integrand given by \eqref{Psi(s) definition}. 
	Choose the positive integer $M\geqslant \max\{1,\Re(-a)\}$ and shift the integration contour to the left 
	(which is permissible on account of the exponential decay of the integrand).
	
	Note that $s=-a-n\ (n\in\ZZ_{\geqslant 0})$ are the only poles of $\Psi(s)$. Therefore,
	\[
	\underset{s=-a-n}{\mathrm{Res}}\Psi(s)
	={}_1F_1\left[\begin{matrix}
		-n\\
		c
	\end{matrix};x\right]\frac{\Gamma(a+n)}{\Gamma(c'-a-n)}\frac{(-1)^n}{n!}(-y)^{-a-n}.
	\]
	By Cauchy's residue theorem, we obtain
	\begin{align*}
		I& =\frac{\Gamma(c')}{\Gamma(a)}\sum_{n=0}^M{}_1F_1\left[\begin{matrix}
			-n\\
			c
		\end{matrix};x\right]\frac{\Gamma(a+n)}{\Gamma(c'-a-n)}\frac{(-1)^n}{n!}(-y)^{-a-n}
		+\underbrace{\frac{1}{2\pi\mi}\frac{\Gamma(c')}{\Gamma(a)}\int_{C_M}\Psi(s)\md s}_{\text{denoted by }R_M(y)}\\
		& =\frac{\Gamma(c')}{\Gamma(c'-a)}\sum_{n=0}^M{}_1F_1\left[\begin{matrix}
			-n\\
			c
		\end{matrix};x\right]\frac{(a)_n(1+a-c')_n}{n!}(-y)^{-a-n}+R_M(y),
	\end{align*}
	where $C_M$ denotes the vertical line $\Re(s)=\Re(-a)-M-1/2$. On account of \eqref{estimate for Psi(s)--1} and 
	\eqref{estimate for Psi(s)--2}, we find that the integral $R_M(y)$ is convergent and that 
	$R_M(y)=\mo_M(|y|^{-\Re(a)-M-\frac{1}{2}})$, which completes the proof.
\end{proof}

We can also obtain the asymptotics of $\Psi_2[x,y]$ as $y\to\infty$ in the right half-plane.

\begin{theorem}\label{Thm: Psi2 large x and y}
	Let $\Re(a)>0$ and $c,c'\in\CC\sm\ZZ_{\leqslant 0}$. Fix $\delta\in\bigl(0,\frac{\pi}{2}\bigr]$ and let
	\[S_{\delta}:=\left\{z\in\CC: z\ne 0,\ \left|\arg(z)\right|\leqslant \frac{\pi}{2}-\delta\right\}.\]
	\quad{\rm(i)} As $t\to+\infty$, for $x,y\in S_{\delta}$ with $0<K_1\leqslant |x|,|y|\leqslant K_2<\infty$,
	\begin{subequations}
		\begin{align}\label{Psi_2[tx,ty] for large t}
			\Psi_2[a;c,c';tx,ty]\sim \frac{\Gamma(c)\Gamma(c')}{2\sqrt{\pi}\,\Gamma(a)}x^{\frac{1}{4}-\frac{1}{2}c}y^{\frac{1}{4}-\frac{1}{2}c'}
			\left(\sqrt{x}+\sqrt{y}\right)^{2a-c-c'}t^{a-c-c'+\frac{1}{2}}\me^{t(\sqrt{x}+\sqrt{y})^2}.
		\end{align}
		\quad{\rm(ii)} For fixed $x\in\CC\sm\{0\}$, as $y\to\infty$ with $y\in S_{\delta}$,
		\begin{align}\label{Psi_2[x,y] for large y in right half-plane}
			\Psi_2[a;c,c';x,y]\sim \frac{\Gamma(c)\Gamma(c')}{2\sqrt{\pi}\,\Gamma(a)} 
			x^{\frac{1}{4}-\frac{1}{2}c}y^{\frac{1}{4}-\frac{1}{2}c'}\left(\sqrt{x}+\sqrt{y}\right)^{2a-c-c'}\me^{(\sqrt{x}+\sqrt{y})^2}.
		\end{align}
	\end{subequations}
\end{theorem}

\begin{proof}
	Let us start with the integral representation \cite[Eq. (2.4c)]{WaHe}
	\[\Psi_2[a;c,c';x,y]=\frac{1}{\Gamma(a)}\int_0^{\infty}s^{a-1}\me^{-s}\,
	{}_0F_1\left[\begin{matrix}
		-\\
		c
	\end{matrix};xs\right]
	{}_0F_1\left[\begin{matrix}
		-\\
		c'
	\end{matrix};ys\right]\md s.
	\]
	Using \eqref{0F1 identity} and setting $s=w^2$, we obtain for $P_t(x,y):=\Psi_2[a;c,c';t^2 x^2,t^2 y^2]$ that
	\begin{equation}\label{Psi_2 another integral}
		P_t(x,y)=\frac{2\,\Gamma(c)\Gamma(c')}{\Gamma(a)}x^{1-c}y^{1-c'}t^{2-c-c'}\int_0^{\infty}w^{2a-c-c'+1}
		\me^{-w^2}I_{c-1}\left(2xtw\right)I_{c'-1}\left(2ytw\right)\md w,
	\end{equation}
	where we assume for the moment that $t,x,y$ are positive.
	
	(i) Denote $\alpha=2x(x+y),\,\beta=2y(x+y)$ and $\tau=t(x+y)$. Set $w=\tau u$ in \eqref{Psi_2 another integral} and then
	\begin{equation}\label{P_t(x,y) integral}
		P_t(x,y)=\frac{2\,\Gamma(c)\Gamma(c')}{\Gamma(a)}x^{1-c}y^{1-c'}t^{2-c-c'}\tau^{2a-c-c'+2}
		\int_0^{\infty}u^{2a-c-c'+1}\me^{-\tau^2 u^2}I_{c-1}\left(\alpha t^2 u\right)I_{c'-1}\left(\beta t^2 u\right)\md u,
	\end{equation}
	which, by analytic continuation, is valid for $t>0$ and $x^2,y^2\in S_{\delta}$.
	
	Denote by $\mathfrak{I}$ the integral on the right side of \eqref{P_t(x,y) integral}. 
	Following the proof of Theorem \ref{Thm: Psi_1[x,y/x] for small x}, we have
	\begin{align*}
		\mathfrak{I} & =\left(\int_0^{t^{-2}}+\int_{t^{-2}}^{\infty}\right)u^{2a-c-c'+1}\me^{-\tau^2 u^2}I_{c-1}\left(\alpha t^2 u\right)I_{c'-1}\left(\beta t^2 u\right)\md u\\
		& \sim \int_{t^{-2}}^{\infty}u^{2a-c-c'+1}\me^{-\tau^2 u^2}I_{c-1}\left(\alpha t^2 u\right)I_{c'-1}\left(\beta t^2 u\right)\md u\\
		& \sim \frac{t^{-2}\me^{\tau^2}}{2\pi\sqrt{\alpha\beta}}\int_{t^{-2}}^{\infty}u^{2a-c-c'+1}
		\me^{-\tau^2\left(u-1\right)^2}\md u\sim \frac{t^{-2}\tau^{-1}\me^{\tau^2}}{2\sqrt{\pi\alpha\beta}},
	\end{align*}
	where in the last line we used \eqref{I_nu for large z} and Olver's Laplace method \cite[Theorem I]{Olve}. Hence, as $t\to+\infty$, for $x,y\in S_{\delta}$ with $0<K_1\leqslant |x|,|y|\leqslant K_2<\infty$,
	\begin{align*}
		P_t(x,y)& \sim \frac{\Gamma(c)\Gamma(c')}{\sqrt{\pi\alpha\beta}\,\Gamma(a)}x^{1-c}y^{1-c'}t^{-c-c'}\tau^{2a-c-c'+1}\me^{\tau^2}\\
		& =\frac{\Gamma(c)\Gamma(c')}{2\sqrt{\pi}\,\Gamma(a)}
		x^{\frac{1}{2}-c}y^{\frac{1}{2}-c'}\left(x+y\right)^{2a-c-c'}t^{2(a-c-c')+1}\me^{t^2(x+y)^2},
	\end{align*}
	which is consistent with the desired result \eqref{Psi_2[tx,ty] for large t}.
	
	(ii) Denote $\lambda=x+y$. Taking $t=1$ and setting $w=\lambda v$ in \eqref{Psi_2 another integral}, we have
	\begin{equation}\label{Psi_2 another integral - 1}
		P_1(x,y)=\frac{2\,\Gamma(c)\Gamma(c')}{\Gamma(a)}x^{1-c}y^{1-c'}\lambda^{2a-c-c'+2}
		\int_0^{\infty}v^{2a-c-c'+1}\me^{-\lambda^2 v^2}I_{c-1}\left(2x\lambda v\right)I_{c'-1}\left(2y\lambda v\right)\md v,
	\end{equation}
	which, by analytic continuation, is valid for fixed $x\in\CC\sm\{0\}$ and large $y$ with $y^2\in S_{\delta}$. 
	
	Denote by $\mathfrak{T}$ the integral on the right side of \eqref{Psi_2 another integral - 1}. Similar to the analysis in (i), we obtain
	\begin{align*}
		\mathfrak{T}& =\left(\int_0^{\left|\lambda\right|^{-2}}
		+\int_{\left|\lambda\right|^{-2}}^{\left|\lambda\right|^{-1}}+\int_{\left|\lambda\right|^{-1}}^{\infty}\right) 
		v^{2a-c-c'+1}\me^{-\lambda^2 v^2}I_{c-1}\left(2x\lambda v\right)I_{c'-1}\left(2y\lambda v\right)\md v\\
		& \sim \int_{\left|\lambda\right|^{-1}}^{\infty}v^{2a-c-c'+1}\me^{-\lambda^2 v^2}I_{c-1}\left(2x\lambda v\right)I_{c'-1}\left(2y\lambda v\right)\md v\\
		& \sim \frac{1}{4\pi}x^{-\frac{1}{2}}y^{-\frac{1}{2}}\lambda^{-1}
		\me^{\lambda^2}\int_{\left|\lambda\right|^{-1}}^{\infty}v^{2a-c-c'}\me^{-\lambda^2\left(v-1\right)^2}\md v
		\sim \frac{1}{4\sqrt{\pi}}x^{-\frac{1}{2}}y^{-\frac{1}{2}}\lambda^{-2}\me^{\lambda^2}.
	\end{align*}
	Hence, for fixed $x\in\CC\sm\{0\}$, as $y\to\infty$ with $y\in S_{\delta}$,
	\[P_1(x,y)\sim \frac{\Gamma(c)\Gamma(c')}{2\sqrt{\pi}\,\Gamma(a)}x^{\frac{1}{2}-c}y^{\frac{1}{2}-c'}\lambda^{2a-c-c'}\me^{\lambda^2}
	=\frac{\Gamma(c)\Gamma(c')}{2\sqrt{\pi}\,\Gamma(a)}x^{\frac{1}{2}-c}y^{\frac{1}{2}-c'}\left(x+y\right)^{2a-c-c'}\me^{\left(x+y\right)^2},\]
	which yields the leading-term behavior shown in \eqref{Psi_2[x,y] for large y in right half-plane}.
\end{proof}

Before proceeding to the next result, we first state a useful lemma, which can be easily 
obtained by repeating  the proof of \cite[Lemma 2.3]{HaLu2}.

\begin{lemma}
	For $z$ bounded away from the points $-b+1,\cdots,-b+n$,
	\begin{equation*}%\label{Pohhammer symbol--lower bound}
		\bigl|(z+b-1)\cdots(z+b-n)\bigr|^{-1}\ll \lambda_n^{-1},
	\end{equation*}
	where
	\begin{align}\label{lambda_n definition}
		\lambda_n&:=\begin{cases}
			1,& 0\leqslant n\leqslant 2,\\
			\min\limits_{1\leqslant k\leqslant n-1}(k-1)!(n-k-1)!\,,& n\geqslant 3
		\end{cases}\notag\\
		&=\begin{cases}
			1,& 0\leqslant n\leqslant 2,\\
			(m-1)!\,m!\,,& n=2m+1\geqslant 3,\\
			\left.(m-1)!\right.^2,& n=2m\geqslant 3.
		\end{cases}
	\end{align}
\end{lemma}

\begin{theorem}\label{Thm: 2F2 expansion}
	Assume that $a,b\in\CC,\,c,d\in\CC\sm\ZZ_{\geqslant 0}$ and $a-b\in\CC\sm\ZZ$. Let $w>0$ be a number such that
	\[w>\max\bigl\{\Re(a),\Re(b),\Re(d)\bigr\}\]
	and that the fractional parts of $w-\Re(a)$ and $w-\Re(b)$ are both in the interval $(\ve,1)$, 
	where $\ve>0$ is a small number. Then for any $n\in\ZZ_{\geqslant 0}$,
	\begin{equation}\label{2F2 expansion for -n}
		{}_2F_2\left[\begin{matrix}
			a,b-n\\
			c,\ \ \,d\ \ \,
		\end{matrix};-z\right]=\frac{\Gamma(c)\Gamma(d)}{\Gamma(a)\Gamma(b-n)}\left\{S_n(z)+T_n(z)+R_{n,w}(z)\right\}
	\end{equation}
	as $z\to\infty$ such that $|\arg(z)|<\pi$ and $z$ is bounded away from the points $-b+k\,(k\in\ZZ)$, where
	\begin{subequations}
		\begin{align}
			S_n(z)&=\sum_{k=0}^{\lfloor w-\Re(a)\rfloor}\frac{\Gamma(a+k)\Gamma(b-a-n-k)}{\Gamma(c-a-k)\Gamma(d-a-k)}\frac{\left(-1\right)^k}{k!}z^{-a-k},\\
			T_n(z)& =\sum_{k=0}^{\lfloor w-\Re(b)\rfloor+n}\frac{\Gamma(b-n+k)\Gamma(a-b+n-k)}{\Gamma(c-b+n-k)\Gamma(d-b+n-k)}\frac{\left(-1\right)^k}{k!}z^{n-b-k},
		\end{align}
		and
		\begin{align}\label{R_{n,w} estimate}
			R_{n,w}(z)&=\mo\left(\lambda_n^{-1}\left\{\left|z\right|^{-w}+\left|z\right|^{\Re(a+b-c-d)}\me^{-\Re(z)}\right\}\right)
		\end{align}
	\end{subequations}
	with $\lambda_n$ defined in \eqref{lambda_n definition}.
\end{theorem}

\begin{proof}
	The proof is much akin to that of \cite[Theorem 2.4]{HaLu2}, so we show only the outline of the proof.
	
	First of all, we use Cauchy's residue theorem to obtain expansion \eqref{2F2 expansion for -n}, 
	where the remainder term $R_{n,w}(z)$ is given explicitly by
	\[
	R_{n,w}(z)=\frac{1}{2\pi\mi}\int_{\mathcal{C}}h_n(s)z^s\md s,~~~
	h_n(s):=\frac{\Gamma(a+s)\Gamma(b-n+s)}{\Gamma(c+s)\Gamma(d+s)}\Gamma(-s).
	\]
	Here $\mathcal{C}$ is a negative-oriented loop that consists of the vertical line
	\[L^v:\quad s=-w+\mi t,\quad |t|\leqslant T\]
	and the contours $L^{\pm}$ which pass to infinity in the directions $\pm \theta_0\,(0<\theta_0<\frac{\pi}{2})$. 
	Furthermore, $\mathcal{C}$ is taken to embrace all the poles of $h_n(s)$ in the right of $L^v$, and $L^{\pm}$ 
	are taken to be bounded away from the poles of $h_n(s)/\Gamma(-s)$ in the right of $L^v$.
	
	Then, divide $R_{n,w}(z)$ into two parts:
	\[R_{n,w}(z)=\frac{1}{2\pi\mi}\int_{L^v}h_n(s)z^s\md s+\frac{1}{2\pi\mi}\int_{L^{\pm}}h_n(s)z^s\md s=:R_v(z)+R_{\pm}(z).\]
	Note that for $s\in\mathcal{C}$,
	\[\left|h_n(s)\right|=\left|h_0(s)\right|\cdot\bigl|(s+b-1)\cdots(s+b-n)\bigr|^{-1}\ll \left|h_0(s)\right|\lambda_n^{-1}.\]
	It follows from \cite[Eq. (2.4) and (2.6)]{LiWo} that
	\[R_v(z)\ll \lambda_n^{-1}\int_{-w-\mi\infty}^{-w+\mi\infty}\left|h_0(s)z^s\right|\left|\md s\right|\ll \lambda_n^{-1}\left|z\right|^{-w}.\]
	According to \cite[pp. 5--6]{HaLu2}, we also have
	\[R_{\pm}(z)\ll\lambda_n^{-1}\int_{L^{\pm}}\left|h_0(s)z^s\right|\left|\md s\right|\ll \lambda_n^{-1}\left|z\right|^{\Re(a+b-c-d)}\me^{-\Re(z)}.\]
	
	Finally, combining the estimates on $R_v(z)$ and $R_{\pm}(z)$ we obtain the remainder estimate \eqref{R_{n,w} estimate}.
\end{proof}

As a direct consequence of Theorem \ref{Thm: 2F2 expansion}, 
we have the following corollary. 

\begin{corollary}\label{Coro: 1F1 uniform estimate}
	Assume that $b\in\CC$ and $d\in\CC\sm\ZZ_{\leqslant 0}$. Let $N$ be an integer such that $N>\max\{|b|,|b-d|\}$. Then for any $n\in\ZZ_{\geqslant 0}$,
	\begin{equation}
		{}_1F_1\left[\begin{matrix}
			b-n\\
			\ \ \,d\ \ \,
		\end{matrix};-z\right]=\frac{\Gamma(d)}{\Gamma(b-n)}\sum_{k=0}^{n+N}\frac{\Gamma(b-n+k)}{\Gamma(d-b+n-k)}\frac{\left(-1\right)^k}{k!}z^{n-b-k}+R(n;z)
	\end{equation}
	as $z\to\infty$ such that $\left|\arg(z)\right|<\pi$ and $z$ is bounded away from the points $-b+k\,(k\in\ZZ)$, where
	\[R(n;z)\ll \left(n+1\right)^{-\Re(b)}n!\,\lambda_n^{-1}\left\{\left|z\right|^{-\Re(b)-N-\frac{1}{2}}+\left|z\right|^{\Re(b-d)}\me^{-\Re(z)}\right\}.\]
\end{corollary}

\begin{proof}
	It suffices to verify the error term here. Take $a=c=\frac{1}{2}$ and $w-\Re(b)=N+\frac{1}{2}$ in \eqref{2F2 expansion for -n}. Then
	\[R(n;z)=\frac{\Gamma(d)}{\Gamma(b-n)}R_{n,w}(z)=\left(-1\right)^n n!\,\frac{\Gamma(d)}{\Gamma(b)}\frac{\left(1-b\right)_n}{n!}R_{n,w}(z).\]
	The estimate on $R(n;z)$ follows from \eqref{R_{n,w} estimate} and \cite[Lemma 2.1]{HaLu1}.
\end{proof}

We are in a position to derive the asymptotics of $\Psi_2$ under the condition \eqref{initial asymptotic condition}.

\begin{theorem}\label{Thm: Psi_2[x,beta/x] for large x}
	Assume that $a\in\CC$ and $c,c'\in\CC\sm\ZZ_{\leqslant 0}$. Set $xy=\beta$. Then under the condition that
	\[x\to\infty,\,\left|\arg(x)\right|<\pi;\quad x\text{ is bounded away from }a-c+k\,(k\in\ZZ);
	\quad 0<\beta_1\leqslant \left|\beta\right|\leqslant\beta_2<\infty,\]
	the function $\Psi_2\equiv \Psi_2[a;c,c';x,y]$ admits the asymptotic expansion
	\begin{align}\label{Thm: Psi_2[x,beta/x] for large x - 1}
		\Psi_2={}& \frac{\Gamma(c)}{\Gamma(c-a)}x^{-a}\sum_{m=0}^{N-1}b_1(m)x^{-m}+
		\frac{\Gamma(c)}{\Gamma(a)}x^{a-c}\me^x\left(\sum_{m=0}^{N-1}b_2(m)x^{-m}+\sum_{m=1}^{N-1}b_3(m)x^{-m}\right)\notag\\
		& +\mo\left(\left|x\right|^{-\Re(a)-N}+\left|x\right|^{\Re(a-c)-N-\frac{1}{2}}\me^{\Re(x)}\right),
	\end{align}
	where $N$ is any positive integer such that $N>\max\{1,|a|,|a-c|\}$, and
	\begin{subequations}
		\begin{align}
			b_1(m)& =\sum_{\substack{n,k\geqslant 0\\2n+k=m}}
			\frac{\left(a\right)_{n+k}\left(a-c+1\right)_{n+k}}{\left(c'\right)_n n!\,k!}\left(-\beta\right)^n,\\
			b_2(m)& ={}_3F_4\left[\begin{matrix}
				1,a+m,a-c+1+m\\
				a,a-c+1,1+m,c'+m
			\end{matrix};\beta\right]\frac{\left(a\right)_m\left(a-c+1\right)_m}{\left(c'\right)_m \left.m!\right.^2}\beta^m,\\
			b_3(m)& =\frac{1}{m!}\sum_{n=0}^{m-1}\frac{\left(1-a-n\right)_m\left(c-a-n\right)_m}{\left(c'\right)_n n!}\beta^n.
		\end{align}
	\end{subequations}
\end{theorem}

\begin{proof}
	For $0\leqslant n\leqslant N$, we get from \eqref{1F1 asymptotics} that
	\begin{equation}\label{1F1 expansion--1}
		{}_1F_1\left[\begin{matrix}
			c-a-n\\
			c
		\end{matrix};-x\right]=f_n(x)+g_n(x)+\mo\left(\left|x\right|^{\Re(a-c)-N-\frac{1}{2}}+\left|x\right|^{-\Re(a)-n-N}\me^{-\Re(x)}\right),
	\end{equation}
	where
	\begin{align*}
		f_n(x)& =\frac{\Gamma(c)}{\Gamma(a+n)}\sum_{k=0}^{n+N}\frac{\left(c-a-n\right)_k\left(1-a-n\right)_k}{k!}x^{a-c+n-k},\\
		g_n(x)& =\frac{\Gamma(c)}{\Gamma(c-a-n)}\me^{-x}\sum_{k=0}^{N-1}\frac{\left(a+n\right)_k\left(a-c+n+1\right)_k}{k!}x^{-a-n-k}.
	\end{align*}
	For $n\geqslant N+1$, Corollary \ref{Coro: 1F1 uniform estimate} gives that
	\begin{equation}\label{1F1 expansion--2}
		{}_1F_1\left[\begin{matrix}
			c-a-n\\
			c
		\end{matrix};-x\right]=f_n(x)+\mo\left(n^{\Re(a-c)}n!\,\lambda_n^{-1}\left\{\left|x\right|^{\Re(a-c)-N-\frac{1}{2}}
		+\left|x\right|^{-\Re(a)}\me^{-\Re(x)}\right\}\right).
	\end{equation}
	Write $P(x,y)=\me^{-x}\Psi_2[a;c,c';x,y]$. It follows from \eqref{Psi_2 series expansion--2}, \eqref{1F1 expansion--1} and \eqref{1F1 expansion--2} that
	\begin{align*}
		P(x,y)
		&=\left(\sum_{n=0}^{N}+\sum_{n=N+1}^{\infty}\right)\frac{\left(a\right)_n}{\left(c'\right)_n}\,_1F_1\left[\begin{matrix}
			c-a-n\\
			c
		\end{matrix};-x\right]\frac{y^n}{n!}\\ &=\sum_{n=0}^{\infty}\frac{\left(a\right)_n}{\left(c'\right)_n}f_n(x)\frac{y^n}{n!}
		+\sum_{n=0}^{N}\frac{\left(a\right)_n}{\left(c'\right)_n}g_n(x)\frac{y^n}{n!}
		+R(x,y)\\
		& =:F(x,y)+G(x,y)+R(x,y).
	\end{align*}
	
	Next, we shall first deal with the remainder $R(x,y)$ and then study the functions $F(x,y)$ and $G(x,y)$ by the series manipulation technique.
	
	\begin{itemize}
		\item[(1)] \textit{Estimate of $R(x,y)$}. The error introduced by \eqref{1F1 expansion--1} is clearly
		\[R_1(x,y)=\mo\left(\left|x\right|^{\Re(a-c)-N-\frac{1}{2}}+\left|x\right|^{-\Re(a)-N}\me^{-\Re(x)}\right),\]
		whereas the error introduced by \eqref{1F1 expansion--2} is
		\[R_2(x,y)\ll \left(\left|x\right|^{\Re(a-c)-N-\frac{1}{2}}+\left|x\right|^{-\Re(a)}\me^{-\Re(x)}\right)
		\sum_{n\geqslant N+1}n^{\Re(a-c)}n!\,\lambda_n^{-1}\left|\frac{\left(a\right)_n}{\left(c'\right)_n}\right|\frac{\left|y\right|^n}{n!}.\]
		Using \cite[Lemma 2.1]{HaLu1} and recalling \eqref{lambda_n definition}, we have
		\begin{align*}
			R_2(x,y)& \ll \left(\left|x\right|^{\Re(a-c)-N-\frac{1}{2}}+\left|x\right|^{-\Re(a)}\me^{-\Re(x)}\right)
			\sum_{n\geqslant N+1}n^{\Re(2a-c-c')}\lambda_n^{-1}\beta_2^n\left|x\right|^{-n}\\
			& \ll \left|x\right|^{\Re(a-c)-2N-\frac{3}{2}}+\left|x\right|^{-\Re(a)-N-1}\me^{-\Re(x)}.
		\end{align*}
		Therefore, the remainder is
		\[R(x,y)=R_1(x,y)+R_2(x,y)=\mo\left(\left|x\right|^{\Re(a-c)-N-\frac{1}{2}}+\left|x\right|^{-\Re(a)-N}\me^{-\Re(x)}\right).\]
		\item[(2)] \textit{Analysis of $F(x,y)$ and $G(x,y)$.} Bear in mind that $xy=\beta$. We have
		\begin{align*}
			G(x,y)& =\frac{\Gamma(c)}{\Gamma(c-a)}x^{-a}\me^{-x}\sum_{n=0}^{N}\frac{\left(a\right)_n\left(a-c+1\right)_n}{\left(c'\right)_n n!}(-\beta)^n x^{-2n}
			\sum_{k=0}^{N-1}\frac{\left(a+n\right)_k\left(a-c+n+1\right)_k}{k!}x^{-k}\\
			& =\frac{\Gamma(c)}{\Gamma(c-a)}x^{-a}\me^{-x}\sum_{m=0}^{N-1}x^{-m}\sum_{\substack{n,k\geqslant 0\\2n+k=m}}\frac{\left(a\right)_n\left(a-c+1\right)_n\left(a+n\right)_k
				\left(a-c+n+1\right)_k}{\left(c'\right)_n n!\,k!}(-\beta)^n\\
			& \quad+\mo\left(\left|x\right|^{-\Re(a)-N}\me^{-\Re(x)}\right)\\
			& =\frac{\Gamma(c)}{\Gamma(c-a)}x^{-a}\me^{-x}\sum_{m=0}^{N-1}b_1(m)x^{-m}+\mo\left(\left|x\right|^{-\Re(a)-N}\me^{-\Re(x)}\right).
		\end{align*}
		For $F(x,y)$, we have
		\begin{align*}
			F(x,y)& =\frac{\Gamma(c)}{\Gamma(a)}x^{a-c}\sum_{n=0}^{\infty}\sum_{k=0}^{n+N}
			\frac{\left(c-a-n\right)_k\left(1-a-n\right)_k}{\left(c'\right)_n n!\,k!}\beta^n x^{-k}\\
			& =\frac{\Gamma(c)}{\Gamma(a)}x^{a-c}\left(\sum_{n=0}^{\infty}\sum_{k=0}^n+\sum_{n=0}^{\infty}\sum_{k=n+1}^{n+N}\right)
			\frac{\left(c-a-n\right)_k\left(1-a-n\right)_k}{\left(c'\right)_n n!\,k!}\beta^n x^{-k}\\
			& =:\frac{\Gamma(c)}{\Gamma(a)}x^{a-c}\bigl(F^-(x,y)+F^+(x,y)\bigr),
		\end{align*}
		where
		\begin{align*}
			F^-(x,y)& =\sum_{n=0}^{\infty}\sum_{k=0}^n \frac{\left(c-a-n\right)_k\left(1-a-n\right)_k}{\left(c'\right)_n n!\,k!}\beta^n x^{-k}\\
			& =\sum_{n=0}^{\infty}\sum_{k=0}^{\infty}\frac{\left(c-a-n-k\right)_k\left(1-a-n-k\right)_k}{\left(c'\right)_{n+k}(n+k)!\,k!}\beta^{n+k}x^{-k}\\
			& =\sum_{k=0}^{\infty}{}_3F_4\left[\begin{matrix}
				1,a+k,a-c+1+k\\
				a,a-c+1,1+k,c'+k
			\end{matrix};\beta\right]\frac{\left(a\right)_k\left(a-c+1\right)_k}{\left(c'\right)_k \left.k!\right.^2}\beta^k x^{-k}
		\end{align*}
		and
		\begin{align*}
			F^+(x,y)& =\sum_{n=0}^{\infty}\sum_{k=n+1}^{n+N}\frac{\left(c-a-n\right)_k\left(1-a-n\right)_k}{\left(c'\right)_n n!\,k!}\beta^n x^{-k}\\
			& =\sum_{n=0}^{N-1}\sum_{j=1}^{N-1}\frac{\left(c-a-n\right)_{n+j}\left(1-a-n\right)_{n+j}}{\left(c'\right)_n n!\,(n+j)!}\beta^n x^{-n-j}
			+\mo\left(\left|x\right|^{-N}\right)\\
			& =\sum_{m=1}^{N-1}\frac{x^{-m}}{m!}\sum_{n=0}^{m-1}\frac{\left(1-a-n\right)_m\left(c-a-n\right)_m}{\left(c'\right)_n n!}\beta^n
			+\mo\left(\left|x\right|^{-N}\right)\\
			&=\sum_{m=1}^{N-1}b_3(m)x^{-m}+\mo\left(\left|x\right|^{-N}\right).
		\end{align*}
		It is clear that the leading terms of the expansion \eqref{Thm: Psi_2[x,beta/x] for large x - 1} are obtained from $F(x,y)$ and $G(x,y)$.
	\end{itemize}
	
	Combining the above expansions gives the asymptotic expansion \eqref{Thm: Psi_2[x,beta/x] for large x - 1} for $\Psi_2$.
\end{proof}

%%###########################################################################################################################%%
\section{Appell function $F_2$}
%%###########################################################################################################################%%

Recall that $F_2$ satisfies the symmetry relation $F_2[a,b,b';c,c';x,y]=F_2[a,b',b;c',c;y,x]$. 
So it suffices to study the asymptotic expansion of $F_2\big[\frac{x}{y},y\big]$ as $y\rightarrow0$.

The following expression for $F_2$ is due to Jaeger (see, for example, \cite[Eq. (9)]{ABFMP},  \cite{Jaeger-1938} and \cite[p. 294]{SrKa}).
\begin{theorem}[Jaeger]
	Let $a,b,b'\in\CC,\,c,c'\in\CC\sm\ZZ_{\leqslant 0}$ and $a-b\notin\ZZ$. Then for $|y|<1$ and $|x|>|y|+1$,
	\begin{equation}\label{F_2 continuation}
		\begin{split}
			F_2[a,b,b';c,c';x,y]={}& \mathfrak{f}_c(b,a)\left(-x\right)^{-a}\sum_{n=0}^{\infty}{}_2F_1\left[\begin{matrix}
				-n,b'\\
				c'
			\end{matrix};y\right]\frac{\left(a\right)_n\left(1-c+a\right)_n}{\left(1-b+a\right)_n n!}\frac{1}{x^n}\\
			&+\mathfrak{f}_c(a,b)\left(-x\right)^{-b}\sum_{n=0}^{\infty}{}_2F_1\left[\begin{matrix}
				a-b-n,b'\\
				c'
			\end{matrix};y\right]\frac{\left(b\right)_n\left(1-c+b\right)_n}{\left(1-a+b\right)_n n!}\frac{1}{x^n},
		\end{split}
	\end{equation}
	where $\mathfrak{f}_{\gamma}(a,b)$ is given by \eqref{f_gamma(a,b) definition}.
\end{theorem}

Using a similar approach to the proof of Theorem \ref{Thm: Psi_1[x/y,y] asymptotics}, 
we obtain from Jaeger's formula \eqref{F_2 continuation} the following result.
\begin{corollary}\label{Cor: F2 small y}
	Let $a,b,b'\in\CC,\,c,c'\in\CC\sm\ZZ_{\leqslant 0}$ and $a-b\notin\ZZ$. For fixed $x\in\CC\sm\{0\}$,
	\begin{subequations}
		\begin{align}
			F_2\left[a,b,b';c,c';\frac{x}{y},y\right]
			={}& \mathfrak{f}_c(b,a)\left(-\frac{x}{y}\right)^{-a}\,\sum_{m=0}^{N-1}c_1(m)y^m+\mathfrak{f}_c(a,b)\left(-\frac{x}{y}\right)^{-b}\,
			\sum_{m=0}^{N-1}c_2(m)y^m\\
			& +\mo\left(\left|y\right|^{\Re(a)+N}+\left|y\right|^{\Re(b)+N}\right) \nonumber 
		\end{align}
		as $y\to 0$ in $\bigl|\arg\bigl(-\frac{x}{y}\bigr)\bigr|<\pi$, where $N$ is any positive integer, $\mathfrak{f}_{\gamma}(a,b)$ 
		is given by \eqref{f_gamma(a,b) definition} and
		\begin{align}
			c_1(m)& =\sum_{\substack{k,\ell\geqslant 0\\k+\ell=m}}\frac{\left(-k\right)_{\ell}(b')_{\ell}
				\left(a\right)_k\left(1-c+a\right)_k}{\left(c'\right)_{\ell}\left(1-b+a\right)_k \ell!\,k!}x^{-k},\\
			c_2(m)& =\sum_{\substack{k,\ell\geqslant 0\\k+\ell=m}}\frac{(b')_{\ell}\left(a-b-k\right)_{\ell}
				\left(b\right)_k\left(1-c+b\right)_k}{\left(c'\right)_{\ell}\left(1-a+b\right)_k \ell!\,k!}x^{-k}.
		\end{align}
	\end{subequations}
\end{corollary}

%%###########################################################################################################################%%
\section{Two methods}\label{Sect: methods}
%%###########################################################################################################################%%
In this section, we propose two elementary methods, which are easy to use and powerful for deducing the asymptotics of multiple hypergeometric functions.

\subsection{Uniformity approach}\label{Sect: uniformity approach}
Our proof of Theorem \ref{Thm: Psi_2[x,beta/x] for large x} is based on the approach used in \cite{HaLu2}, 
which is indeed effective for deriving the asymptotics of multiple hypergeometric functions. Let us give a detailed description:
\begin{itemize}
	\item (\textbf{Uniformity approach})
	{\it
		For $n\in\ZZ_{\geqslant 0}$, denote
		\begin{equation}\label{F_n(z) definition}
			F_n(z)={}_pF_q\left[\begin{matrix}
				a_1+\lambda_1 n,\cdots,a_r+\lambda_r n,a_{r+1},\cdots,a_p\\
				b_1+\mu_1 n,\cdots, b_s+\mu_s n,b_{s+1},\cdots,b_q
			\end{matrix};z\right],
		\end{equation}
		where $\lambda_1,\cdots,\lambda_r,\mu_1,\cdots,\mu_s$ are fixed complex numbers. In view of {\rm\cite[Sect. 5]{LiWo}}, for fixed $n$,
		\[F_n(z)\sim A_n(z)+E_n(z),\quad z\to\infty,\]
		where $A_n(z)$ is the algebraic expansion and $E_n(z)$ is the exponential expansion.
		
		\quad To derive the asymptotics of the function
		\begin{equation}\label{F(z,w) definition}
			\mathcal{F}(z,w):=\sum_{n=0}^{\infty}a_n F_n(z)w^n,
		\end{equation}
		we first deduce the estimate for $F_n(z)$, that is akin to \eqref{2F2 expansion for -n} and valid uniformly for $n\geqslant 0$, namely,
		\begin{equation}\label{F_n(z) estimate}
			F_n(z)=A_n^{(N)}(z)+E_n^{(N)}(z)+\mo\bigl(c_n r(z)\bigr),\quad n\in\ZZ_{\geqslant 0},\,|z|\geqslant K,
		\end{equation}
		where $A_n^{(N)}(z)$ (resp. $E_n^{(N)}(z)$) is the sum of first $N$ terms of $A_n(z)$ (resp. $E_n(z)$). 
		Then inserting \eqref{F_n(z) estimate} into \eqref{F(z,w) definition} 
		and using the series manipulation technique, we get the asymptotic expansion of $\mathcal{F}(z,w)$.
	}
\end{itemize}

Let us present the powerful utility of the uniformity approach. Very recently, Brychkov and Savischenko 
systematically studied the formulas of the Horn functions 
$H_1,\cdots,H_7$ and the confluent Horn functions $H_1^{(c)},\cdots,H_{11}^{(c)}$. 
For example, in \cite{BrSav1, BrSav3, BrSav4}, they obtained the series expansions below:
\begin{align*}
	H_2^{(c)}[a,b,c;d;x,y]& =\sum_{n=0}^{\infty}\frac{\left(c\right)_n}{\left(1-a\right)_n}\,_2F_1\left[\begin{matrix}
		a-n,b\\
		d
	\end{matrix};x\right]\frac{(-y)^n}{n!},\\
	H_3^{(c)}[a,b;d;x,y]& =\sum_{n=0}^{\infty}\frac{1}{\left(1-a\right)_n}\,_2F_1\left[\begin{matrix}
		a-n,b\\
		d
	\end{matrix};x\right]\frac{(-y)^n}{n!},\\
	H_4^{(c)}[a,c;d;x,y]& =\sum_{n=0}^{\infty}\frac{\left(c\right)_n}{\left(1-a\right)_n}\,_1F_1\left[\begin{matrix}
		a-n\\
		d
	\end{matrix};x\right]\frac{(-y)^n}{n!},\\
	H_5^{(c)}[a;c;x,y]& =\sum_{n=0}^{\infty}\frac{1}{\left(1-a\right)_n}\,_1F_1\left[\begin{matrix}
		a-n\\
		c
	\end{matrix};x\right]\frac{(-y)^n}{n!}.
\end{align*}
By repeating the proof of Theorem \ref{Thm: 2F2 expansion}, one can obtain an expansion of ${}_2F_1[a-n,b;c;x]$, 
which is similar to \eqref{2F2 expansion for -n}. Then the use of the uniformity approach will give the full asymptotic expansions of 
$H_k^{(c)}\big[x,\frac{\beta}{x}\big]$ $(k=2,3,4,5)$ for large $x$. Furthermore, in view of the series expansions derived in \cite{AnDG, BrSav2}, 
the same apply to the Appell functions $F_1,F_3$ and the Humbert functions $\Phi_1,\Phi_2,\Phi_3,\Xi_1$.

\subsection{Separation method}
We provide a heuristic approach, called \textbf{separation method}, to reproduce \eqref{Psi_1 series for |x|>1} and \eqref{F_2 continuation}. 
This method is of great benefit to physicists, although it is not rigorous in general.

\begin{proposition}\label{Prop: Psi_1 for large x}
	Under the necessary conditions on $a,b,c,c'$, for $x\to +\infty$ and fixed $y>0$, one has
	\begin{equation}\label{Psi_1 series reproduced}
		\begin{split}
			\Psi_1[a,b;c,c';-x,y]\sim{}&
			\mathfrak{f}_{c}(b,a) x^{-a}\sum_{k=0}^{\infty}{}_1F_1\left[\begin{matrix}
				-k\\
				c'
			\end{matrix};y\right]\frac{\left(a\right)_k\left(1-c+a\right)_k}{\left(1-b+a\right)_k k!}\frac{1}{\left(-x\right)^k}\\
			&+\mathfrak{f}_{c}(a,b)x^{-b}\sum_{\ell=0}^{\infty}{}_1F_1\left[\begin{matrix}
				a-b-\ell\\
				c'
			\end{matrix};y\right]\frac{\left(b\right)_{\ell}\left(1-c+b\right)_{\ell}}{\left(1-a+b\right)_{\ell}\ell!}\frac{1}{\left(-x\right)^{\ell}},
		\end{split}
	\end{equation}
	where $\mathfrak{f}_{\gamma}(a,b)$ is defined by \eqref{f_gamma(a,b) definition}. 
	See also Remark~\ref{rem3.7} after replacing $y$ by $\frac{y}{x}$ in (\ref{Psi_1 series reproduced}).
\end{proposition}

\begin{proof}
	Our starting point is the integral representation \eqref{Psi_1 Laplace integral}. For large $x$, we split the integral into two parts:
	\begin{align*} \Psi_1[a,b;c,c';-x,y]&=\frac{1}{\Gamma(a)}\left(\int_0^{\eta/x}+\int_{\eta/x}^{\infty}\right)u^{a-1}\me^{-u}\,_1F_1\left[\begin{matrix}
			b\\
			c
		\end{matrix};-xu\right]
		{}_0F_1\left[\begin{matrix}
			-\\
			c'
		\end{matrix};yu\right]\md u\\
		&=:\mathfrak{T}_1^{(\eta)}+\mathfrak{T}_2^{(\eta)},
	\end{align*}
	where $\eta$ is taken such that both $\eta\to+\infty$ and $\eta=o(x)$ hold as $x\to+\infty$. 
	
	We begin with the integral $\mathfrak{T}_1^{(\eta)}$. Setting $v=ux$, we obtain 
	\[
	\mathfrak{T}_1^{(\eta)}=\frac{x^{-a}}{\Gamma(a)}\int_{0}^{\eta}v^{a-1}
	\mathrm{e}^{-\frac{v}{x}}{}_{1}F_{1}\left[\begin{matrix}
		b\\
		c
	\end{matrix};-v\right]
	{}_{0}F_{1}\left[\begin{matrix}
		-\\
		c'
	\end{matrix};\frac{vy}{x}\right]\md v.
	\] 
	Since
	\begin{align*}
		\me^{-\frac{v}{x}}\,_0F_1\left[\begin{matrix}
			-\\
			c'
		\end{matrix};\frac{vy}{x}\right]
		&=\sum_{n=0}^{\infty}\frac{\left(-1\right)^n}{n!}\left(\frac{v}{x}\right)^n
		\sum_{m=0}^{\infty}\frac{v^m}{\left(c'\right)_m m!}\left(\frac{y}{x}\right)^m\\
		&=\sum_{k=0}^{\infty}\left(\frac{v}{x}\right)^k\sum_{m=0}^k\frac{\left(-1\right)^{k-m} y^m}{\left(c'\right)_m m!\,(k-m)!}\\
		&=\sum_{k=0}^{\infty}\left(\frac{v}{x}\right)^k\frac{\left(-1\right)^k}{k!}\sum_{m=0}^k\frac{\left(-k\right)_m}{\left(c'\right)_m}\frac{y^m}{m!}
		=\sum_{k=0}^{\infty}\left(\frac{v}{x}\right)^k\frac{\left(-1\right)^k}{k!}\,{}_1F_1\left[\begin{matrix}
			-k\\
			c'
		\end{matrix};y\right],
	\end{align*}
	we have
	\begin{align*}
		\mathfrak{T}_1^{(\eta)}
		& =\frac{x^{-a}}{\Gamma(a)}\sum_{k=0}^{\infty}\frac{\left(-x\right)^{-k}}{k!}\,{}_1F_1\left[\begin{matrix}
			-k\\
			c'
		\end{matrix};y\right]\int_0^{\eta}v^{a-1+k}\,{}_1F_1\left[\begin{matrix}
			b\\
			c
		\end{matrix};-v\right]\md v.
	\end{align*}
	
	To proceed, we notice that \cite[p. 408, Eq. (3.28.1.1)]{BrMS}
	\[\int_0^{\infty}t^{s-1}\,_1F_1\left[\begin{matrix}
		a\\
		b
	\end{matrix};-t\right]\md t=\frac{\Gamma(s)\Gamma(b)\Gamma(a-s)}{\Gamma(a)\Gamma(b-s)},\quad 0<\Re(s)<\Re(a).	
	\]
	By formally letting $x\to+\infty$, we have 
	\begin{align*}
		\mathfrak{T}_1^{(\eta)}& \sim \frac{x^{-a}}{\Gamma(a)}\sum_{k=0}^{\infty}\frac{\left(-x\right)^{-k}}{k!}\,
		{}_1F_1\left[\begin{matrix}
			-k\\
			c'
		\end{matrix};y\right]
		\int_0^{\infty}v^{a-1+k}\,
		{}_1F_1\left[\begin{matrix}
			b\\
			c
		\end{matrix};-v\right]\md v\\
		& =\frac{x^{-a}}{\Gamma(a)}\sum_{k=0}^{\infty}\frac{\left(-x\right)^{-k}}{k!}\,{}_1F_1\left[\begin{matrix}
			-k\\
			c'
		\end{matrix};y\right]\frac{\Gamma(a+k)\Gamma(b-a-k)\Gamma(c)}{\Gamma(c-a-k)\Gamma(b)}\\
		& =\frac{\Gamma(c)\Gamma(b-a)}{\Gamma(b)\Gamma(c-a)}x^{-a}\sum_{k=0}^{\infty}{}_1F_1\left[\begin{matrix}
			-k\\
			c'
		\end{matrix};y\right]\frac{\left(a\right)_k\left(1-c+a\right)_k}{\left(1-b+a\right)_k k!}\frac{1}{\left(-x\right)^k},
	\end{align*}
	where in the last line we used the identity
	\begin{equation}\label{Gamma(z-k) identity}
		\Gamma(z-k)=\left(-1\right)^k\frac{\Gamma(z)}{\left(1-z\right)_k}.
	\end{equation}
	
	We now turn to the integral $\mathfrak{T}_2^{(\eta)}$ and use a formal derivation to obtain an expansion. 
	Since the ${}_{1}F_{1}$ function in the integrand has a negative argument, the algebraic expansion should dominate, which inspires us to write 
	\begin{align*}
		\mathfrak{T}_2^{(\eta)}& \sim \frac{1}{\Gamma(a)}\int_{\eta/x}^{\infty}u^{a-1}\me^{-u}\,
		{}_0F_1\left[\begin{matrix}
			-\\
			c'
		\end{matrix};yu\right]\frac{\Gamma(c)}{\Gamma(c-b)}\sum_{n=0}^{\infty}\frac{\left(b\right)_n\left(1+b-c\right)_n}{n!}\left(xu\right)^{-b-n}\md u\\
		& =\frac{\Gamma(c)\Gamma(c')}{\Gamma(a)\Gamma(c-b)}
		\frac{y^{\frac{1}{2}(1-c')}}{x^b}
		\sum_{n=0}^{\infty}\frac{\left(b\right)_n\left(1+b-c\right)_n}{n!}x^{-n}
		\int_{\eta/x}^{\infty}\me^{-u}u^{a-b-\frac{1}{2}(c'+1)-n}I_{c'-1}(2\sqrt{yu})\md u,
	\end{align*}
	where we used \eqref{0F1 identity}. By formally letting $x\rightarrow+\infty$ and taking into account that $\eta=o(x)$, we get
	\begin{align*}
		\mathfrak{T}_2^{(\eta)}& \sim \frac{\Gamma(c)\Gamma(c')}{\Gamma(a)\Gamma(c-b)}
		\frac{y^{\frac{1}{2}(1-c')}}{x^b}\sum_{n=0}^{\infty}\frac{\left(b\right)_n\left(1+b-c\right)_n}{n!}x^{-n}
		\int_0^{\infty}\me^{-u}u^{a-b-\frac{1}{2}(c'+1)-n}I_{c'-1}(2\sqrt{yu})\md u\\
		& =\frac{2\,\Gamma(c)\Gamma(c')}{\Gamma(a)\Gamma(c-b)}\frac{y^{\frac{1}{2}(1-c')}}{x^b}
		\sum_{n=0}^{\infty}\frac{\left(b\right)_n\left(1+b-c\right)_n}{n!}x^{-n}\int_0^{\infty}\me^{-w^2}
		w^{2a-2b-c'-2n}I_{c'-1}(2\sqrt{y}\,w)\,
		\md w\\
		& \sim\frac{\Gamma(c)}{\Gamma(a)\Gamma(c-b)}x^{-b}\sum_{n=0}^{\infty}\frac{\left(b\right)_n\left(1+b-c\right)_n}{n!}\Gamma(a-b-n)
		\,{}_1F_1\left[\begin{matrix}
			a-b-n\\
			c'
		\end{matrix};y\right]
		x^{-n},
	\end{align*}
	where we used the Mellin transform \cite[Eq. (3.13.2.3)]{BrMS}
	\[
	\int_0^{\infty}t^{s-1}\me^{-at^2}I_{\nu}\left(bt\right)\md t=\frac{2^{-\nu-1}b^{\nu}}{a^{\frac{1}{2}(s+\nu)}}
	\frac{\Gamma\big(\frac{1}{2}(s+\nu)\big)}{\Gamma(\nu+1)}\,_1F_1\left[\begin{matrix}
		\frac{1}{2}(s+\nu)\\
		\nu+1
	\end{matrix};\frac{b^2}{4a}\right],\quad
	\Re(a)>0, \Re(s+\nu)>0.
	\]
	Then use \eqref{Gamma(z-k) identity} to yield
	\[\mathfrak{T}_2^{(\eta)}\sim\frac{\Gamma(c)\Gamma(a-b)}{\Gamma(a)\Gamma(c-b)}x^{-b}\sum_{n=0}^{\infty}{}_1F_1\left[\begin{matrix}
		a-b-n\\
		c'
	\end{matrix};y\right]\frac{\left(b\right)_n\left(1-c+b\right)_n}{\left(1-a+b\right)_n n!}\frac{1}{\left(-x\right)^n},\]
	which is the second term in \eqref{Psi_1 series reproduced}.
\end{proof}

\begin{remark}
	Starting with \eqref{Psi_1 Laplace integral} and applying the separation method, one may rigorously deduce the leading asymptotics of 
	$\Psi_1[x,y]$ for $x<1$ and $y\to\infty$ in $\left|\arg(y)\right|<\frac{\pi}{2}$, which was given in {\rm\cite[Section 3.1]{HaLu3}}.
\end{remark}

\begin{proposition}
	Under the necessary conditions on $a,b,b',c,c'$, for $x\to +\infty$ and fixed $y<1$, one has
	\begin{align}
		F_2[a,b,b';c,c';-x,y]\sim{}&
		\mathfrak{f}_{c}(b,a) x^{-a}\sum_{n=0}^{\infty}{}_2F_1\left[\begin{matrix}
			-n,b'\\
			c'
		\end{matrix};y\right]\frac{\left(a\right)_n\left(1-c+a\right)_n}{\left(1-b+a\right)_n n!}\frac{1}{\left(-x\right)^n} \nonumber \\
		&+\mathfrak{f}_{c}(a,b)x^{-b}\sum_{n=0}^{\infty}{}_2F_1\left[\begin{matrix}
			a-b-n,b'\\
			c'
		\end{matrix};y\right]\frac{\left(b\right)_n\left(1-c+b\right)_n}{\left(1-a+b\right)_n n!}\frac{1}{\left(-x\right)^n},
	\end{align}
	where $\mathfrak{f}_{\gamma}(a,b)$ is defined by \eqref{f_gamma(a,b) definition}. 
\end{proposition}

\begin{proof}
	Let us start with the following Laplace-type integral representation of $F_2$:
	\begin{equation*}%\label{F_2 Laplace integral}
		F_2[a,b,b';c,c';x,y]
		=\int_0^{\infty}t^{a-1}\me^{-t}\,
		{}_1F_1\left[
		\begin{matrix}
			b\\
			c
		\end{matrix}
		;xt\right]
		{}_1F_1\left[
		\begin{matrix}
			b'\\
			c'\end{matrix};yt\right]\md t,
	\end{equation*}
	which, in view of \eqref{1F1 asymptotics}, is valid for $\Re(a)>0$ and $\Re(x+y)<1$. 
	As in the proof of Proposition \ref{Prop: Psi_1 for large x}, we split the integral into two parts: 
	\begin{align*}
		F_2[a,b,b';c,c';-x,y]
		&=\left(\int_0^{\eta/x}+\int_{\eta/x}^{\infty}\right)t^{a-1}\me^{-t}\,
		{}_1F_1\left[
		\begin{matrix}
			b\\
			c
		\end{matrix}
		;-xt\right]
		{}_1F_1\left[
		\begin{matrix}
			b'\\
			c'\end{matrix};yt\right]\md t\\
		&=:\mathfrak{T}_1^{(\eta)}+\mathfrak{T}_2^{(\eta)},
	\end{align*}
	where $\eta>0$ is taken such that both $\eta\to+\infty$ and $\eta=o(x)$ hold as $x\to+\infty$.
	
	Recall that \cite[Eq. (3.28.2.1)]{BrMS}
	\[
	\int_0^{\infty}t^{s-1}\me^{-\sigma t}\,_1F_1\left[\begin{matrix}
		a\\
		b
	\end{matrix};\omega t\right]\md t=\frac{\Gamma(s)}{\sigma^s}\,_2F_1\left[\begin{matrix}
		a,s\\
		b
	\end{matrix};\frac{\omega}{\sigma}\right].
	\]
	We follow the proof of Proposition \ref{Prop: Psi_1 for large x} and obtain that as $x\to+\infty$,
	\begin{align*}
		\mathfrak{T}_1^{(\eta)}& :=\frac{1}{\Gamma(a)}\int_0^{\eta/x}u^{a-1}\me^{-u}\,_1F_1\left[\begin{matrix}
			b\\
			c
		\end{matrix};-xu\right]
		{}_1F_1\left[\begin{matrix}
			b'\\
			c'
		\end{matrix};yu\right]\md u\\
		& =\frac{x^{-a}}{\Gamma(a)}\sum_{k=0}^{\infty}\frac{\left(-x\right)^{-k}}{k!}\,_2F_1\left[\begin{matrix}
			-k,b'\\
			c'
		\end{matrix};y\right]\int_0^{\eta}v^{a+k-1}\,_1F_1\left[\begin{matrix}
			b\\
			c
		\end{matrix};-v\right]\md v\\
		& \sim \frac{x^{-a}}{\Gamma(a)}\sum_{k=0}^{\infty}\frac{\left(-x\right)^{-k}}{k!}\,_2F_1\left[\begin{matrix}
			-k,b'\\
			c'
		\end{matrix};y\right]\int_0^{\infty}v^{a+k-1}\,_1F_1\left[\begin{matrix}
			b\\
			c
		\end{matrix};-v\right]\md v\\
		& =\frac{\Gamma(c)\Gamma(b-a)}{\Gamma(b)\Gamma(c-a)}x^{-a}\sum_{n=0}^{\infty}{}_2F_1\left[\begin{matrix}
			-n,b'\\
			c'
		\end{matrix};y\right]\frac{\left(a\right)_n\left(1-c+a\right)_n}{\left(1-b+a\right)_n n!}\frac{1}{\left(-x\right)^n},\\
		\mathfrak{T}_2^{(\eta)}& :=\frac{1}{\Gamma(a)}\int_{\eta/x}^{\infty}u^{a-1}\me^{-u}\,_1F_1\left[\begin{matrix}
			b\\
			c
		\end{matrix};-xu\right]{}_1F_1\left[\begin{matrix}
			b'\\
			c'
		\end{matrix};yu\right]\md u\\
		& \sim \frac{1}{\Gamma(a)}\int_{\eta/x}^{\infty}
		u^{a-1}\me^{-u}\,_1F_1\left[\begin{matrix}
			b'\\
			c'
		\end{matrix};yu\right]\frac{\Gamma(c)}{\Gamma(c-b)}\sum_{n=0}^{\infty}\frac{\left(b\right)_n\left(1+b-c\right)_n}{n!}\left(xu\right)^{-b-n}\md u\\
		& \sim \frac{\Gamma(c)}{\Gamma(a)\Gamma(c-b)}x^{-b}\sum_{n=0}^{\infty}\frac{\left(b\right)_n\left(1+b-c\right)_n}{n!}x^{-n}
		\int_0^{\infty}\me^{-u}u^{a-b-n-1}\,_1F_1\left[
		\begin{matrix}
			b'\\
			c'
		\end{matrix};yu\right]\md u\\
		& \sim \frac{\Gamma(c)\Gamma(a-b)}{\Gamma(a)\Gamma(c-b)}x^{-b}\sum_{n=0}^{\infty}{}_2F_1\left[\begin{matrix}
			a-b-n,b'\\
			c'
		\end{matrix};y\right]\frac{\left(b\right)_n\left(1-c+b\right)_n}{\left(1-a+b\right)_n n!}\frac{1}{\left(-x\right)^n}.
	\end{align*}
	Combining these results gives the desired formula.
\end{proof}

%%###########################################################################################################################%%
\section{Concluding remarks}
%%###########################################################################################################################%%

In this paper, we deduced the asymptotics of the Humbert functions $\Psi_1,\Psi_2$ and the Appell function $F_2$ under the condition \eqref{initial asymptotic condition}. 
Our methods should be generalizable to find similarly the asymptotics of the Humbert functions $\Phi_1, \Phi_2, \Phi_3, \Xi_1, \Xi_2$ and the Appell functions $F_1,F_3$; see Remark \ref{Remark: asymptotics of F_3 and Humbert functions} and Section \ref{Sect: uniformity approach}.

We also derived the asymptotics of $\Psi_2$ for one or two large arguments. As an application, our result \eqref{Psi_2[tx,ty] for large t} corrects Saran's estimate on $\Psi_2$ \cite[Section 8, Eq. (1)]{Saran-1957}
\[\Psi_2[a,c;c';x,y]\sim C\,\me^{x+y+2\sqrt{xy}},\quad x,y\to \infty.\]
Recall that Saran's function $F_E$ admits a Laplace-type integral \cite[p. 134, Eq. (1)]{Saran-1957}
\begin{align*}
	&F_E[a_1,a_1,a_1,b_1,b_2,b_2;c_1,c_2,c_3;x,y,z]\\
	&\hspace{1cm}=\frac{1}{\Gamma(\alpha_1)}\int_0^{\infty}\me^{-t}t^{a_1-1}
	{}_1F_1\left[\begin{matrix}
		b_1\\
		c_1
	\end{matrix};xt\right]
	\Psi_2\left[b_2;c_2,c_3;yt,zt\right]\md t,
\end{align*}
but Saran's convergence condition
\[\Re\left(x+y+z+2\sqrt{yz}\right)<1,\ \Re(a_1)>0\]
is wrong. Our result \eqref{Psi_2[tx,ty] for large t} provides a sufficient condition
\[\Re(y)>0,\ \Re(z)>0,\ \Re\left(y+z+2\sqrt{yz}\right)<1+\min\bigl\{0,-\Re(x)\bigr\},\ \Re(a_1)>0.\]
To clarify the general convergence condition, we are led to establish in later work the asymptotics of $\Psi_2[tx,ty]$ as $t\to +\infty$ with $x,y$ in general case.

In addition, we made initial attempts on two asymptotic techniques we proposed,
i.e., the uniformity approach and the separation method, both of which are currently deficient.

The uniformity approach is based on the series expansions involving the generalized hypergeometirc function $F_n(z)$ 
(see \eqref{F_n(z) definition}), and the main difficulty is to explicitly deduce the uniform estimate \eqref{F_n(z) estimate}. 
We have dealt with the simplest cases:
\begin{enumerate}[\quad(i)]
	\item Estimate on ${}_2F_2[a,b-n;c,d-n;-z]$: see \cite[Theorem 2.4]{HaLu2};
	
	\item Estimate on ${}_2F_2[a,b-n;c,d;-z]$: see Theorem \ref{Thm: 2F2 expansion};
	
	\item Estimate on ${}_2F_1[a,b-n;c;-z]$: briefly mentioned in Section \ref{Sect: methods}.
\end{enumerate}
Unfortunately, exponential expansions are implicit in the estimates for (i) and (ii), 
since nice integrals are missing for such ${}_2F_2$ functions.
Please note Blaschke's conjecture \cite[p. 1791]{Blas}:
\begin{itemize}
	\item (\textbf{More down conjecture})
	{\it
		Denote $\mathscr{F}(\alpha,z)$ by
		\[\mathscr{F}(\alpha,z):={}_pF_q\left[\begin{matrix}
			a_1\pm \alpha,\cdots,a_k\pm \alpha,a_{k+1},\cdots,a_p\\
			b_1\pm \alpha,\cdots,b_m\pm \alpha,b_{m+1},\cdots,b_q
		\end{matrix};z\right].\]
		Regardless of the sign, when more large parameters are down than up, the resulting Taylor series defining $\mathscr{F}(\alpha,z)$ 
		is always an asymptotic expansion for some values of $z$. That is, if $k<m$, then
		\[\mathscr{F}(\alpha,z)\sim \sum_{n=0}^{\infty}\frac{\left(a_1\pm\alpha\right)_n
			\cdots\left(a_k\pm\alpha\right)_n\left(a_{k+1}\right)_n\cdots\left(a_p\right)_n}{\left(b_1\pm\alpha\right)_n
			\cdots\left(b_m\pm\alpha\right)_n\left(b_{m+1}\right)_n\cdots\left(b_q\right)_n}\frac{z^n}{n!},
		\quad \left|\alpha\right|\to\infty.\]
	}
\end{itemize}
The uniformity approach is indeed as difficult as the more down conjecture, since they both are somewhat ill-posed problems. 
It would be interesting to see whether the steepest descent method might become helpful for these problems.

The separation method is based on the Mellin convolution integrals. 
Another asymptotic technique for integrals of this type is the sum up and subtract (SUS) 
method introduced by L\'opez \cite{Lope}. L\'opez and his coauthors \cite{FeLS} used the SUS method to derive the asymptotics of $F_2[x,y]$ for large $x$, 
whereas our separation method works as well. In addition, the SUS method fails for $\Psi_1$, but our separation method still works 
(see Proposition \ref{Prop: Psi_1 for large x}). These facts illustrate that the separation method is pretty powerful, though it is just heuristic. 
It would be desirable to make this method rigorous.

Remarkably, we could not obtain the asymptotics of $\Psi_1[x,y]$ (see \cite{HaLu3}) and $\Psi_2[x,y]$ as $y\to\infty$ 
with arbitrary values of parameters, because we lack of useful expansions of ${}_2F_1$ and ${}_1F_1$, which are expanded in terms of large parameters. 
If such effective expansions of ${}_pF_q$ were known, one might use the Mellin-Barnes integral 
technique to establish the asymptotics of multiple hypergeometric functions for one large argument.

%%###########################################################################################################################%%
\section*{Acknowledgment}
%%###########################################################################################################################%%
%We are very grateful to the anonymous referee for his/her careful reading and valuable suggestions which have helped to improve the presentation of this paper.

We are grateful to Petr Blaschke, Yury Brychkov and Nico Temme for their valuable advice during personal discussions with the first author. The second author was supported by the french ANR-PRME UNIOPEN (ANR-22-CE30-0004-01).

\end{document}